\documentclass{article}[12pt]
\newtheorem{theorem}{Theorem}

\newtheorem{corollary}[theorem]{Corollary}

\newtheorem{remark}[theorem]{Remark}

\newenvironment{proof}[1][Proof]{\noindent\textbf{#1.} }{\ \rule{0.5em}{0.5em}}
\usepackage{amssymb,amsmath}
\usepackage{graphicx}
\usepackage[dvipsnames,usenames]{color}
\usepackage{hyperref}
\input{epsf.tex}
\usepackage[affil-it]{authblk}

\author{Ioannis Dimitriou \footnote{ idimit@uoi.gr}\footnote{Corresponding author.}}
\affil{\small Department of Mathematics, 
	University of Ioannina, 
	45110, Ioannina, Greece.}

\begin{document}
\title{On a Class of Multiplicative Lindley-type Recursions with Markov-Modulated Dependencies}

\maketitle
\begin{abstract}
 In this paper, we study a class of Markov-modulated multiplicative Lindley-type recursions of the form $W_{n+1}=[V_nW_n+Y_n(V_n)]^+$, where the additive term may depend on the multiplicative coefficient and on a common finite-state background Markov chain. The main feature of the model is that the multiplicative coefficient is allowed to take negative values, which destroys
the usual monotonicity structure of Lindley-type recursions. We consider two cases. In the first one, \(V_n\) may take the values \(1\), a constant
\(a\in(0,1)\), or a negative value, according to specific probabilities, while \(Y_n(V_n)=S_n-A_{n+1}\) and the relevant
Laplace--Stieltjes transforms are assumed to be rational. In the second one, \(V_n\in\{1,-1\}\), and the additive term has a more general form depending on \(V_n\). For both models, we derive vector-valued functional equations for the stationary transform. For the first model, we obtain a recursive transform representation, while for the second model we derive a finite-dimensional characterization based on root-cancellation conditions. We also provide a recursive procedure for computing stationary moments and discuss the light-tailed asymptotic behaviour. An extension based on the Farlie--Gumbel--Morgenstern copula is also presented. Numerical experiments show how the transform representation can be used to compute stationary moments, tail probabilities and quantiles, and validate the
results against Monte Carlo simulation.
    \end{abstract}
    \vspace{2mm}
	

\maketitle
\section{Introduction}
\label{sec1}
In this work, we focus on the stationary analysis of the Markov-modulated multiplicative version of the Lindley-type recursion
\begin{equation}
     W_{n+1}=[V_{n}W_{n}+Y_{n}(V_{n})]^{+}, \label{recu1}
\end{equation}
with $[x]^{+}:=max\{x,0\}$. We assume that $Y_{n}(V_{n})$ can be written as a difference of two nonnegative random variables that may depend on $V_{n}$. Stochastic recursions have received much attention in the applied probability
literature, and especially those of autoregressive type,
due to their wide applicability across various scientific domains including biology, finance, and engineering, e.g., \cite{brandt,bura,emb,collam,le}.

Much of the related literature focuses on stability conditions, limit theorems and queueing applications, whereas our primary focus lies on the derivation of results for the transient and stationary distribution of the process under investigation. Furthermore, most works in this area concern the scalar case. Our focus is on the derivation of the stationary behaviour of vector-valued stochastic recursions of the type in \eqref{recu1}, which in turn result in vector-valued functional equations.  

In this work, we focus on two models. In Model I, the $V_n$ are either equal to one, or equal to a positive constant $a\in(0,1)$, or negative, according to certain probabilities. We further assume that $Y_{n}(V_{n})=S_{n}-A_{n+1}$ (i.e., independent of $V_{n}$), and that the conditional Laplace–Stieltjes transforms (LST) of the nonnegative random variables
$S_n$ and $A_{n+1}$ are rational. Note that in such a case, $S_{n}$, $A_{n+1}$ are regulated by a finite state space irreducible Markov chain, thus, given the state of the background Markov chain, $S_{n}$, $A_{n+1}$ are conditionally independent.  In Model II,
the $V_n$ are equal either to one, or to minus one, and $Y_{n}(V_{n})$ have a more general form that is dependent on $V_{n}$. An extension allowing Farlie-Gumbel-Morgenstern (FGM) copula \cite{nelsen} dependence between the two nonnegative components is discussed in Appendix \ref{secfgm}. We focus on FGM copula due to its
polynomial structure that leads to tractable and analytic results.

%% Labels are used to cross-reference an item using \ref command.

%% Use \subsection commands to start a subsection.
\subsection{Related work}
\label{subsec1}
In \cite{box1}, the authors considered the (scalar) recursion $W_{n+1}=[aW_{n}+S_{n}-A_{n}]^{+}$, where $\{S_{n}-A_{n}\}_{n\in\mathbb{N}_{0}}$ forms a sequence of independent and identically
distributed (i.i.d.) random variables and $a\in(0,1)$. The authors investigated both the transient and stationary behaviour. Note that in such a case, $W_{n}$ could be interpreted as the workload in a queueing system just before the $n$th arrival, which adds $S_{n}$ work, and makes obsolete a fixed fraction $1-a$ of the already present work. The case where $a=1$ corresponds to the classical Lindley recursion describing the waiting time of the classical G/G/1 queue \cite{asmussen}. The case where $a=-1$ produces a non-increasing Lindley-type recursion and was studied in the scalar setting in \cite{vlasiou2007non}, including its stationary analysis and stability; see also \cite{vlasiou}. Moreover, for $a=-1$, related scalar recursions arise in alternating-service and carousel models. In particular, in \cite{vlasiou2005alternating} the authors gave the alternating-service interpretation and the phase-type analysis, while the authors in \cite{vlasiou2004lindley} provided the two-carousel interpretation and an explicit stationary solution.

In \cite{box2}, the authors provided transient and the stationary characteristics for the scalar autoregressive process described by the recursion $W_{n+1}=[V_{n}W_{n}+S_{n}-A_{n}]^{+}$, where $V_{n}$ can be positive or negative, while in \cite{box3}, the authors considered the case where $V_{n}W_{n}$ was replaced by $S(W_{n})$, where $\{S(t)\}$ is a Levy subordinator. Motivated by applications that arise in queueing and insurance risk models, the authors in \cite{boxman} considered Lindley-type recursions where $\{B_{n}\}_{n\in\mathbb{N}_{0}}$, $\{A_{n}\}_{n\in\mathbb{N}_{0}}$ obey a semi-linear dependence. The seminal work in \cite{adan} provided a method to study functional equations that arise in a wide range of queueing, autoregressive and branching processes, and where (scalar) Lindley-type recursions arise. A generalized version of the model in \cite{box2}, was investigated in \cite{hoo}. Quite recently, in \cite{dimitriou2024} the authors generalized the work in \cite{box1} by considering, among others, non-trivial dependence structures among $\{S_{n}\}_{n\in\mathbb{N}}$, $\{A_{n}\}_{n\in\mathbb{N}}$. The model in \cite{boxvla} is related to our work but corresponds to the scalar case. For the scalar random-sign Lindley recursion, \cite{vlasiou2010tail} derived tail asymptotics. Their results provided
a useful scalar comparison for Model II and the tail analysis in Section \ref{asym}. In a related continuous-time setting, the authors in \cite{palmowski2011levy} studied a workload process driven by a reflected spectrally positive L\'evy process and modified by
state-dependent services at exponential epochs. Their model provides a related perspective on workload regulation and alternating-service mechanisms; see also \cite{palmowski2014levy} for a related extension. In \cite{vlasiou2007exact}, the authors analyzed the negative-sign scalar recursion and gave an exact stationary solution under a polynomial-distribution assumption, while in \cite{vlasiou2007time} the transient waiting-time distributions and correlations for an alternating-service model were studied using a (scalar) Lindley-type recursion.

Explicit analysis of Markov-modulated (multiplicative) Lindley-type stochastic recursion is scarce in the literature. Quite recently, in \cite{dimitriou2024markov}, the author investigated vector-valued reflected autoregressive processes that are described by vector-valued Lindley type recursions of the form $W_{n+1}=[aW_{n}+S_{n}-A_{n+1}]^{+}$, where the sequences $\{S_{n}\}_{n\in\mathbb{N}_{0}}$, $\{A_{n}\}_{n\in\mathbb{N}_{0}}$ are governed by an irreducible finite state background Markovian process, thus, considered the Markov-dependent version of the process analysed in \cite{box1}. Note that the specific case of $a = 1$ in \cite{dimitriou2024markov} corresponds to the waiting time in a single server queue with Markov-dependent interarrival and service times studied in \cite{adan2}; see also \cite{albbox} for a Markov-dependent structure related in insurance mathematics. The case where $a=-1$ was investigated in \cite{vlasioudep} in the context of carousel models. In \cite{dimitriou2024markov} the focus was on the case where $a\in(0,1)$. Moreover, the author also discussed the case where $\{(S_{n},A_{n+1})\}_{n\in\mathbb{N}_{0}}$ form a sequence of i.i.d. random vectors with a distribution function defined by the Farlie-Gumbel-Morgenstern copula, as well as the case where $\{(S_{n},A_{n+1})\}_{n\in\mathbb{N}_{0}}$ have a bivariate matrix-exponential distribution. In both cases, the joint distribution is also dependent on the state of the background discrete time Markov chain.

%Subsection text.

%% Use \subsubsection, \paragraph, \subparagraph commands to 
%% start 3rd, 4th and 5th level sections.
%% Refer following link for more details.
%% https://en.wikibooks.org/wiki/LaTeX/Document_Structure#Sectioning_commands

\subsection{Contribution}

The main contribution of this paper is a transform-based framework for the stationary analysis of Markov-modulated multiplicative Lindley-type recursions of the form
\[
W_{n+1}=[V_nW_n+Y_n(V_n)]^+,
\]
where the additive term may depend on the value of \(V_n\) and on a finite-state background Markov chain. The distinctive feature of the model is that the multiplicative coefficient is allowed to be negative with positive probability. This destroys the usual monotonicity property of Lindley-type recursions and prevents a direct use of standard monotone-recursion or duality arguments. Our contribution is threefold.

First, we derive vector-valued functional equations for the stationary transform vector of the Markov-modulated recursion. For Model I, where \(V_n\) may take the values \(1\), \(a\in(0,1)\), or a negative value, with certain probabilities, and \(Y_n=S_n-A_{n+1}\), we obtain a recursive transform representation under rationality assumptions on the relevant Laplace-Stieltjes transforms (LST) of the primitive random variables. In doing that we apply Wiener-Hopf arguments, and this is the main reason for choosing \(Y_n=S_n-A_{n+1}\) with rational LST. Furthermore, distributions with rational LSTs form a broad tractable class, including
phase-type and matrix-exponential distributions, and can approximate general
nonnegative distributions arbitrarily closely in the usual weak-convergence sense
\cite[Chapter III]{asmussen}. Thus, although rationality is a genuine tractability assumption,
it covers a flexible and widely used class of models in applied probability. For Model II, where \(V_n\in\{1,-1\}\) and the additive term depends on \(V_n\), we derive a finite-dimensional characterization of the stationary transform. The analysis requires root-cancellation conditions associated with the kernel matrix of the vector-valued transform equation.

Second, we show that the transform representation is not only an analytical
characterization but also a computational tool. For Model II, once the finite set of constants entering the transform has been identified, the moment recursion of Theorem~5 provides a direct method for computing stationary moments and variability measures. In addition, numerical Laplace inversion of the stationary transform gives
access to distributional quantities such as the unconditional distribution function $F_W(x)$, the tail probability $P(W>x)$, and stationary quantiles. Thus, the method provides more than the stationary mean: it yields a computational route to moments, tail probabilities and quantiles. We further use the method to isolate the impact of
Markov modulation by varying the transition matrix of the background chain while keeping its stationary distribution fixed. This shows that temporal dependence in the environment affects not only the stationary mean, but also the variance, coefficient of variation (CV), tail probabilities and upper quantiles. Finally, we present transform-based computations to study the
combined effect of autocorrelation and state-mean cross-correlation in the interarrival
and service-time sequences. These experiments show how dependence affects not only
$E(W)$, but also $Var(W)$ and $CV(W)$.

Third, the two models unify and extend complementary strands of the existing literature. Model I extends scalar reflected autoregressive/Lindley-type recursions with mixed multiplicative coefficients to a Markov-modulated setting. It also extends the Markov-modulated model with deterministic contraction coefficient \(a\in(0,1)\) by allowing \(V_n\) to take the value \(1\), the value \(a\in(0,1)\), or a negative value with positive probability. Thus, Model I combines classical Lindley-type dynamics, contractive autoregressive dynamics and sign-changing dynamics in a Markov-dependent framework.

Model II focuses on the case \(V_n\in\{1,-1\}\), where the additive term depends explicitly on the value of \(V_n\). When \(V_n=1\), the recursion is related to Markov-modulated single-server queueing models \cite{adan2}. When $V_n=-1$, the recursion is related to alternating-service and carousel models \cite{vlasiou2005alternating,vlasiou2004lindley},
as well as to the related two-station model in \cite{vlasioudep}. The case \(p\in(0,1)\) combines these mechanisms in a single Markov-modulated recursion, while the corresponding scalar case was investigated in \cite{boxvla}. We also discuss in Appendix \ref{secfgm} an extension in which the two nonnegative random variables forming \(Y_n(V_n)\) are dependent through a Farlie--Gumbel--Morgenstern (FGM) copula. This shows that the transform approach remains tractable beyond the conditionally independent case.

The class of recursions considered here is motivated by queueing and service systems in which the next workload may depend linearly on the current workload and on an external Markovian environment. More generally, \(W_n\) may represent an inventory or workload level, \(V_n\) a random multiplicative adjustment, and \(Y_n(V_n)\) a state-dependent input-output balance.

%Additional motivation for studying such a recursion is by the fact that, for $p\in(0,1)$, the resulting model 
\subsection{ Structure of the paper}
The paper is organized as follows. In Section \ref{mod} we provide the description of the two models and some preliminaries. Sections \ref{mod1}, \ref{mod2} are devoted to the stationary analysis of Models I, II, respectively. In Section \ref{num}, we provide extensive numerical experiments that illustrate and validate the theoretical findings and offer insights into system performance. Some concluding remarks and suggestions for future research are presented in Section \ref{conc}.
%% Inline mathematics is tagged between $ symbols.
%his is an example for the symbol $\alpha$ tagged as inline mathematics.

%% Displayed equations can be tagged using various environments. 
%% Single line equations can be tagged using the equation environment.
%\begin{equation}
%f(x) = (x+a)(x+b)

%\end{equation}
\section{Model description and preliminaries}\label{mod}
%% Unnumbered equations are tagged using starred versions of the environment.
%% amsmath package needs to be loaded for the starred version of equation environment.
%\begin{equation*}
%f(x) = (x+a)(x+b)
In this work we focus on the stationary analysis of a Markov-dependent version of the recursion \eqref{recu1}, where $\{Y_{n}(V_{n})\}_{n\in\mathbb{N}_{0}}$ is a sequence of random variables that is written as a difference of two random variables. Their distributions are regulated by an irreducible, aperiodic, discrete-time Markov chain $\{Z_{n}\}_{n\in\mathbb{N}_{0}}$ with state space $E=\{1, 2,...,N\}$, one-step transition probability matrix $P:=(p_{i,j})_{i,j\in E}$, and stationary distribution $\tilde{\pi}:=(\pi_{1},\ldots,\pi_{N})^{T}$. Conditional on the relevant states of the background chain, the primitive variables are independent of the past. In the main analysis the two components forming each difference are conditionally independent, while an FGM-dependent extension is treated in Appendix \ref{secfgm}.

Let $\Phi_{W,i}^{n}(s):=E(e^{-sW_{n}}1_{\{Z_{n}=i\}})$, $i=1,\ldots,N,$ $Re(s)\geq 0$, $n\geq 0$, and assuming the limit exists, define $\Phi_{W,i}(s)=\lim_{n\to\infty}\Phi_{W,i}^{n}(s)$, $i=1,\ldots,N$. Let also  $\tilde{\Phi}_{W}(s):=(\Phi_{W,1}(s),\ldots,\Phi_{W,N}(s))^{T}$.

In this paper, we discuss the following two variants of the model in \eqref{recu1}:
\begin{enumerate} 
\item \textit{Model I:} $P(V_{n}=1)=p_{1}$, $P(V_{n}=a)=p_{2}$, $a\in(0,1)$, $P(V_{n}<0)=p_{3}:=1-p_{1}-p_{2}$, $Y_{n}(V_{n}):=Y_{n}=S_{n}-A_{n+1}$. We further assume that for $n\geq 0$, $x,y\geq 0$, $i,j=1,\ldots,N$:
\begin{equation}
    \begin{array}{l}
         P(A_{n+1}\leq x,S_{n}\leq y,Z_{n+1}=j|Z_{n}=i,(C_{r+1},D_{r},Z_{r}),0\leq r\leq n-1) \vspace{2mm}\\
        =  P(A_{n+1}\leq x,S_{n}\leq y,Z_{n+1}=j|Z_{n}=i) \vspace{2mm}\\=p_{i,j}P(A_{n+1}\leq x,S_{n}\leq y|Z_{n}=i,Z_{n+1}=j)=p_{i,j}F_{S,i}(y)G_{A,j}(x),
    \end{array}\label{mp1}
\end{equation}
where $F_{S,i}(.)$, $G_{A,j}(.)$ denote the distribution functions of $S_{n}$, $A_{n+1}$, given $Z_{n}=i$, $Z_{n+1}=j$, respectively. Note that $A_{n+1}$, $S_{n}$, $Z_{n+1}$ are independent of the past given $Z_{n}$, and $A_{n+1}$, $S_{n}$ are conditionally independent given $Z_{n}$, $Z_{n+1}$. We assume that the sequences $\{A_n\}_{n\in\mathbb{N}}$ and $\{S_n\}_{n\in\mathbb{N}_{0}}$ are autocorrelated as well as cross-correlated. 
\item \textit{Model II:} $P(V_{n}=1)=p$, $P(V_{n}=-1)=q:=1-p$, and
    \begin{displaymath}
    Y_{n}(V_{n})=\left\{\begin{array}{ll}
         S_{n}-A_{n+1},&\text{ when }V_{n}=1,  \\
         D_{n+1}-C_{n},& \text{ when }V_{n}=-1,
    \end{array}\right.
\end{displaymath}
Note that when $V_{n}=1$, the dynamics among $\{A_{n}\}_{n\in\mathbb{N}}$, $\{S_{n}\}_{n\in\mathbb{N}_{0}}$ are governed by \eqref{mp1}. When $V_{n}=-1$, then,
\begin{equation}
    \begin{array}{l}
         P(D_{n+1}\leq x,C_{n}\leq y,Z_{n+1}=j|Z_{n}=i,(D_{r+1},C_{r},Z_{r}),0\leq r\leq n-1) \vspace{2mm}\\
        =  P(D_{n+1}\leq x,C_{n}\leq y,Z_{n+1}=j|Z_{n}=i) \vspace{2mm}\\=p_{i,j}P(D_{n+1}\leq x,C_{n}\leq y|Z_{n}=i,Z_{n+1}=j)=p_{i,j}F_{C,i}(y)G_{D,j}(x),
    \end{array}\label{mp2}
\end{equation}
where $F_{C,i}(.)$, $G_{D,j}(.)$ denote the distribution functions of $C_{n}$, $D_{n+1}$, given $Z_{n}=i$, $Z_{n+1}=j$, respectively. Note that $D_{n+1}$, $C_{n}$, $Z_{n+1}$ are independent of the past given $Z_{n}$, and $D_{n+1}$, $C_{n}$ are conditionally independent given $Z_{n}$, $Z_{n+1}$. We assume that the sequences $\{D_n\}_{n\in\mathbb{N}}$ and $\{C_n\}_{n\in\mathbb{N}_{0}}$ are autocorrelated as well as cross-correlated. In the main analysis we assume conditional independence between the two components forming each difference.

An FGM-dependent extension is discussed separately in Appendix \ref{secfgm}. Under such a setting, given the values of $Z_{n}=i$, $Z_{n+1}=j$, $A_{n+1}$, $S_{n}$, and $D_{n+1}$, $C_{n}$, respectively, are dependent based on the FGM copula that allows for explicit results due to its polynomial structure \cite{nelsen,joe}; see Appendix \ref{secfgm}. With the above, we fill a gap in the literature that is related to the stationary analysis of Markov-modulated multiplicative Lindley-type recursions with dependencies. 
\end{enumerate}
%To investigate the stationary behaviour of $\{W_{n}\}_{n\in\mathbb{N}_{0}}$ we adopt the conditions in \cite[Theorem 1]{gyorfi}; see also \cite[Theorem 5.1]{diaconis}. 
\section{Model I: The Markov-dependent mixed case}\label{mod1}
Starting from the recursion 
\begin{equation}
    W_{n+1}=[V_{n}W_{n}+S_{n}-A_{n+1}]^{+},\label{recu}
\end{equation}
and under the Markov-dependent dynamics in \eqref{mp1}, we assume that $V_{n} = 1$ with probability $p_{1}$, $V_{n}=a\in (0,1)$ with probability $p_{2}$, and $V_{n} < 0$ with probability $p_{3}=1-p_{1}-p_{2}$. We consider the case where $0<p_{i}<1$, $i=1,2,3,$ $p_{1}+p_{2}<1$. If $p_{1}=1$, the model is reduced to the one in \cite{adan2}, while if $p_{2}=1$ is reduced to the one in \cite{dimitriou2024markov}, and if $p_{3}=1$ is reduced to the one in \cite{vlasioudep}.

\begin{remark}(Stability and regeneration mechanism)
We denote by $X$ a generic random variable distributed as $X_0$, and assume throughout that $E(|V|)<\infty$, $E(|Y|)<\infty$. For the scalar recursion, \cite[Theorem~1]{whitt} shows that if $P(V<0)>0$ and $P(Y\leq 0)>0$, then $W_n$ converges in distribution to a proper limiting random variable.

In the present Markov-modulated setting, the same reset mechanism is still available, but it should be formulated for the joint process $(W_n,Z_n)$. Indeed, since $V_n<0$ implies $V_nW_n\leq0
$ for every $W_n\geq0$, we have
\begin{displaymath}
    V_n<0,\quad Y_n=S_n-A_{n+1}\leq0
\qquad\Longrightarrow\qquad
W_{n+1}=0.
\end{displaymath}
Thus, such events force the workload to return to zero independently of its previous value. In our model, $P(V_n<0)=p_3>0$. Moreover, in the special case where, conditionally on $Z_n=i$ and $Z_{n+1}=j$, $A_{n+1}\sim exp(\lambda_j)$, and $S_n$ has distribution function $B_i(\cdot)$ with LST $\beta_i^*(\cdot)$, we have
\begin{displaymath}
    P(S_n-A_{n+1}\leq0\mid Z_n=i,Z_{n+1}=j)= \beta_i^*(\lambda_j).
\end{displaymath}
Consequently,
\begin{displaymath}
    P(Y_n\leq0) = \sum_{i=1}^N\sum_{j=1}^N
\pi_i p_{ij}\beta_i^*(\lambda_j)>0.
\end{displaymath}
Hence the reset event $\{V_n<0,\,Y_n\leq0\}$
has positive probability.

Since the background chain is finite and irreducible, the joint process $(W_n,Z_n)$ returns to the set ${0}\times E$ with positive probability. More precisely, visits to states of the form $(0,j)$, $j\in E$, act as regeneration epochs for the joint process: conditional on $Z_n=j$ and $W_n=0$, the future evolution is independent of the past and has the same law whenever the process returns to $(0,j)$. Under the above reset condition, these regeneration cycles have finite mean length. Therefore, the joint process $(W_n,Z_n)$ admits a unique stationary distribution. In particular, $W_n\Rightarrow W$, as $n\to\infty$, for a proper limiting random variable $W$.
\end{remark}

We assume here on that $S|Z=i$ and $A|Z=j$ have a rational LST. Note that this is a structural tractability condition, and it is used to reduce the Wiener--Hopf-type boundary relation to a finite-dimensional polynomial problem. More precisely, the rational form of the LSTs allows us to identify a finite set of unknown constants through root conditions and interpolation-type equations. In addition, it is known that the distribution of any nonnegative random variable can be approximated arbitrarily closely, in a weak-convergence sense, by the distribution of a random variable with a rational LST \cite[Chapter III]{asmussen}. Thus, although rationality is a genuine tractability assumption,
it covers a flexible and widely used class of models in applied probability including
phase-type and matrix-exponential distributions. Therefore, let
\begin{displaymath}
    \Phi_{B,i}(s):=\int_{0}^{\infty}e^{-st}dF_{S,i}(t):=\frac{N_{B,i}(s)}{D_{B,i}(s)},\,\Phi_{A,j}(s):=\int_{0}^{\infty}e^{-st}dG_{A,j}(t):=\frac{N_{A,j}(s)}{D_{A,j}(s)}.
\end{displaymath}
Here $N_{B,i}(s)$, $D_{B,i}(s)$, $N_{A,j}(s)$, $D_{A,j}(s)$ are polynomials, and after cancellation the zeros of the denominators are denoted by
\begin{displaymath}
    t_{1,i},\ldots,t_{m_i,i},\qquad\text{and}\qquad
s_{1,j},\ldots,s_{l_j,j},
\end{displaymath}
respectively, with $Re(t_{k,i})<0$, $k=1,\ldots,m_i$, $i=1,\ldots,N$, and $Re(s_{r,j})<0$, $r=1,\ldots,l_j$, $j\in E=\{1,\ldots,N\}$.

In such a setting, denote by 
\begin{displaymath}
    \begin{array}{rl}
        H_{i,j}(s):=&E(e^{-sY_{n}}|Z_{n}=i,Z_{n+1}=j)=E(e^{-s(S_{n}-A_{n+1})}|Z_{n}=i,Z_{n+1}=j)\vspace{2mm}\\=&\frac{N_{Y}^{(i,j)}(s)}{D_{Y}^{(i,j)}(s)}=\tilde{\Phi}_{B,i}(s)\tilde{\Phi}_{A,j}(-s)=\frac{N_{B,i}(s)N_{A,j}(-s)}{D_{B,i}(s)D_{A,j}(-s)}, 
    \end{array}
\end{displaymath}
and the $N\times N$ matrix $H(s):=(H_{i,j}(s))_{i,j=1,\ldots,N}$. For $j=1,\ldots,N$, let
\begin{displaymath}
    \Delta_j(s)=D_{A,j}(-s)\prod_{k=1}^{N}D_{B,k}(s),\quad
d_j=l_j+\sum_{k=1}^{N}m_k,
\end{displaymath}
and define
\begin{displaymath}
    \mathcal D_+:=
\left\{
s\in\mathbb C:
\begin{array}{l}
Re(s)\geq 0,\\[1mm]
E(
\left|e^{-s(S_n-A_{n+1})}\right|
|\, Z_n=i,Z_{n+1}=j
)<\infty,\text{ for all }i,j\in E,\\[2mm]
\Delta_j(s)\neq 0,\text{ for }j\in E
\end{array}
\right\}.
\end{displaymath}

Thus, in $\mathcal D_+$ each $H_{i,j}(s)$ is interpreted as its defining expectation, which is absolutely integrable, and
$\Delta_j(s)\neq 0$ for every $j\in E$. The zeros of $D_{A,j}(s)$ are
$s_{r,j}$, with $Re(s_{r,j})<0$, so the zeros of $D_{A,j}(-s)$ are
$-s_{r,j}$ and lie in the right half-plane. The zeros of
$D_{B,k}(s)$ lie in the left half-plane. Thus, the condition
$\Delta_j(s)\neq 0$ excludes the zeros of the denominator factors in
the rational expressions. However, this condition alone, does not
ensure that the defining expectation for $H_{i,j}(s)$ is finite, since a rational continuation beyond the transform's convergence domain need not equal that expectation. This is why expectation finiteness is also required in the definition of $\mathcal D_+$.

Define also
\begin{displaymath}
    \Omega=\{s\in \mathcal{D}_+:det(I-p_1 F(s))\neq 0\},\qquad \Omega_a=\{s\in \Omega:a^{m}s\in \Omega,m\geq 0\}. 
\end{displaymath}
Then, we have the following result.
\begin{theorem}\label{theo1}
For $s\in \mathcal{D}_+$, the transforms $\Phi_{W,j}(s)$, $j = 1,\ldots, N$ satisfy the system
\begin{equation}
    \Phi_{W,j}(s)-\sum_{i=1}^{N}p_{i,j}H_{i,j}(s) \left[ p_{1}\Phi_{W,i}(s)+p_{2}\Phi_{W,i}(as) \right]=\frac{\sum_{\omega=0}^{d_{j}}c_{\omega,j}s^{\omega}}{\Delta_{j}(s)},\label{b0}
\end{equation}
where 
\begin{displaymath}
    c_{0,j}=\pi_{j}p_{3}\Delta_{j}(0),
\end{displaymath}
and $c_{\omega,j}$, $\omega=1,\ldots,d_{j}$, $j=1,\ldots,N$ can be determined from the linear systems \eqref{sys2}, \eqref{sys10} given in Appendix \ref{prooftheo1}. The procedure to obtain the coefficients is stated under the additional simple-root condition \eqref{st}.

In matrix notation, for $s\in\mathcal D_+$ this system is
\begin{equation}
    (I-p_1F(s))\tilde\Phi_W(s)=p_2F(s)\tilde\Phi_W(as)+\tilde v(s).\label{undiv}
\end{equation}
Equivalently, for $s\in\Omega$ it is written as
\begin{equation}
    \tilde{\Phi}_{W}(s)=R(s)\tilde{\Phi}_{W}(as)+\tilde{V}(s),\label{funeq}
\end{equation}
where
\begin{displaymath}
    \begin{array}{rl}
         R(s):=&p_{2}(I-p_{1}F(s))^{-1}F(s),\vspace{2mm} \\
         \tilde{V}(s):=&(I-p_{1}F(s))^{-1}\tilde{v}(s),
    \end{array}
\end{displaymath}
and $F(s):=P^{T}\circ H(s)^{T}$, with '$\circ$' denotes the Hadamard product and 
\begin{displaymath}
    \tilde{v}(s):=(\frac{\sum_{\omega=0}^{d_{1}}c_{\omega,1}s^{\omega}}{\Delta_{1}(s)},\ldots,\frac{\sum_{\omega=0}^{d_{N}}c_{\omega,N}s^{\omega}}{\Delta_{N}(s)})^{T}. 
\end{displaymath}
Then, for $s\in \Omega_a$ the solution to \eqref{funeq} is
\begin{equation}
      \tilde{\Phi}_{W}(s)=\sum_{k=0}^{\infty}\prod_{m=0}^{k-1}R(a^{m}s)\tilde{V}(a^{k}s).
   \label{soel}
\end{equation}
\end{theorem}
\begin{proof}
The proof is given in Appendix \ref{prooftheo1}.
\end{proof}

\begin{remark}
The series representation \eqref{soel} is valid for $s\in\Omega_a$.
Let $s_\star\in\mathcal D_+\setminus\Omega_a$ be an exceptional point
for this representation. By definition, there is an integer $m_0\geq0$
such that $a^{m_0}s_\star\notin\Omega$. If
$a^{m_0}s_\star\in\mathcal D_+$, then it is a kernel root, i.e., $det(I-p_1F(a^{m_0}s_\star))=0$. Otherwise, that iterate fails one of the transform-domain or denominator conditions defining $\mathcal D_+$.

The value $\tilde{\Phi}_W(s_\star)$ is defined by the stationary
transform. It is obtained from \eqref{soel} by taking a limit through
regular points. Choose $s_n\in\Omega_a$ with $s_n\to s_\star$. Then
\[
\tilde{\Phi}_W(s_\star)
=
\lim_{n\to\infty}\tilde{\Phi}_W(s_n)
=
\lim_{n\to\infty}
\sum_{k=0}^{\infty}
\left(\prod_{m=0}^{k-1}R(a^m s_n)\right)
\tilde V(a^k s_n).
\]
For each $n$, the series is evaluated at the regular point $s_n$, and the limit is then taken. The limit is independent of the sequence, since
$\tilde{\Phi}_W(s)$ is analytic for $Re(s)>0$ and continuous for
$Re(s)=0$.

Under the rational-transform assumptions, such regular points can be
chosen arbitrarily close to $s_\star$ within the transform domain.
Indeed, $0\in\Omega$, and the rational factors are regular at $0$.
Thus there is a neighborhood of $0$, relative to the transform domain,
in which the conditions defining $\Omega$ hold. Choose a bounded
neighborhood of admissible points approaching $s_\star$. Since
$0<a<1$, all sufficiently late iterates $a^m s$ lie in that
neighborhood of $0$, uniformly for $s$ in the chosen neighborhood.
The expectation-finiteness condition is preserved when the
nonnegative real part of $s$ is decreased, so it causes no additional
obstruction along these iterates. Only finitely many earlier iterates
need to be considered. The zeros of their rational denominator
factors $\Delta_j(a^m s)$ and of
$\det(I-p_1F(a^m s))$ are isolated, and can therefore be avoided
when choosing $s_n$.

If $r=a^{m_0}s_\star\in\mathcal D_+$ is a kernel root, then
$K(r):=I-p_1F(r)$ is singular. The equation
\eqref{undiv} remains valid at $r$, whereas \eqref{funeq} cannot be evaluated there. Set
\[
b(r):=p_2F(r)\tilde{\Phi}_W(ar)+\tilde v(r).
\]
A necessary compatibility condition for
$K(r)\tilde{\Phi}_W(r)=b(r)$ is
\[
l^{T}b(r)=0,
\]
for every $l$ satisfying $l^{T}K(r)=0$. The root-cancellation conditions in Appendix \ref{prooftheo1}
enforce this compatibility at the kernel roots considered there.
The stationary transform, obtained by the limit above, selects its
value at such a point. Thus, the singular matrix equation alone need not determine that value uniquely. Points at which an expectation
defining $H_{i,j}$ is not finite are outside $\mathcal D_+$ and are
not covered here.
\end{remark}

%\begin{remark}
%For $s_0\in\mathcal D_+\setminus\Omega$, the equation \eqref{funeq} cannot be used, since $I-p_1F(s_0)$ is singular. The
%series representation \eqref{soel} is valid only for
%$s\in\Omega_a$. The vector-valued equation \eqref{undiv} remains valid at $s_0$, and $\tilde{\Phi}_W(s_0)$ is defined by the stationary transform.
%Since this transform is analytic for $Re(s)>0$ and continuous for $Re(s)=0$, its value at an exceptional point is obtained by continuous extension from nearby regular points in $\Omega$ (from within $\mathcal D_+$ at boundary points). The root-cancellation conditions used to determine the constants $c_{\omega,j}$ provide the compatibility required for this extension.
%\end{remark}

Once $\tilde{\Phi}_{W}(s)$ is fully specified in \eqref{soel}, we can obtain steady-state moments of any order. For convenience, we only focus on the mean waiting time, but our framework allows for more general setting. Let $m_{i}:=\lim_{n\to\infty}E(W1_{\{Z_{n}=i\}})$, $i=1,\ldots,N$, and $\tilde{m}=(m_{1},\ldots,m_{N})^{T}$. Differentiating \eqref{b1} with respect to $s$ and letting $s\to 0$ we obtain 
\begin{equation}
    \tilde{m}=((p_{1}+ap_{2})P^{T}-I)^{-1}(\Phi\tilde{\pi}+v^{\prime}(0)),
\end{equation}
where $\Phi=P^{T}\circ (H^{T})^{\prime}(0)$, where $"\prime"$ denotes the first derivative.
\subsection{A special case}
We briefly describe a special case of Model I by assuming $p_{3}=1$, i.e., $V<0$ a.s.. Moreover, we drop the assumption of rationality of $\Phi_{A}(.)$, since now it is easier to handle the resulting functional equation. We prefer to give a separate analysis of this special case, to make the
reader familiar with the specific mathematical intricacies due to $V_n$ being negative and $V_n$ being a positive constant. 

Moreover, the additional assumption that $S_{n,i}$, $A_{n,i}$ have a rational LST plays an important role
in the analysis of Model I, excluding the possibility of simply obtaining the results for this special case directly from those for Model I. Then, following similar arguments as in Theorem \ref{theo1} we obtain for $Re(s)=0$,
\begin{displaymath}
        \Phi_{W,j}^{n+1}(s) =\sum_{i=1}^{N}p_{i,j}\frac{N_{B,i}(s)}{D_{B,i}(s)}\Phi_{A,j}(-s)\int_{-\infty}^{0} \Phi_{W,i}^{n}(sy)P(V_{n}^{-}\in dy)
         +U^{n}_{j}(s).
   \end{displaymath}
We first discuss the stability properties of
the system. We assume that there are $i$ and $j$ such that $P(S_{n}< A_{n+1},Z_{n+1}=l|Z_{n}=i)>0$. This implies
that the Markov chain $\{(W_{n},Z_{n}),n\in\mathbb{N}_{0}\}$ is stable; see also \cite{vlasioudep} for a similar argument. So, since $\{Z_n\}$ is aperiodic, the limiting distribution of $\{(W_{n},Z_{n}),n\in\mathbb{N}_{0}\}$ exists,
and $Re(s) > 0$, $n > 1$, and $j = 1, 2, \ldots,N$, $\Phi_{W,j}^{n}(s)\to\Phi_{W,j}(s)$, as $n\to\infty$. Therefore, for $n\to\infty$, $Re(s)=0$,
\begin{equation}
\begin{array}{rl}
        \Phi_{W,j}(s)\prod_{k=1}^{N}D_{B,k}(s)=&\Phi_{A,j}(-s)\sum_{i=1}^{N}p_{i,j}\prod_{\nu\neq i}^{N}D_{B,\nu}(s)\int_{-\infty}^{0} \Phi_{W,i}(sy)P(V^{-}\in dy)\vspace{2mm}\\&+\prod_{k=1}^{N}D_{B,k}(s)U_{j}(s).\end{array}\label{dg1}
\end{equation}
Then,
\begin{enumerate}
    \item the left-hand side of \eqref{dg1} is analytic in $Re (s) >0$ and continuous in $Re (s) \geq 0$,
    \item the right-hand side of \eqref{dg1} is analytic in $Re (s) <0$ and continuous in $Re (s) \leq 0$,
\item for large $s$, and $j=1,\ldots,N,$ both sides are $\mathcal{O}(s^{\sum_{k=1}^{N}m_{k}})$ in their respective half-planes.
\end{enumerate}
Therefore, Liouville's theorem \cite{tit} implies that for $Re(s)\geq 0$,
\begin{equation}
    \Phi_{W,j}(s)\prod_{k=1}^{N}D_{B,k}(s)=\sum_{\omega=0}^{\sum_{k=1}^{N}m_{k}}c_{\omega,j}s^{\omega},\label{aa1}
\end{equation}
and for $Re(s)\leq 0$,
\begin{equation}\begin{array}{l}
    \Phi_{A,j}(-s)\sum_{i=1}^{N}p_{i,j}\prod_{\nu\neq i}^{N}D_{B,\nu}(s)\int_{-\infty}^{0} \Phi_{W,i}(sy)P(V^{-}\in dy)\vspace{2mm}\\+\prod_{k=1}^{N}D_{B,k}(s)U_{j}(s)=\sum_{\omega=0}^{\sum_{k=1}^{N}m_{k}}c_{\omega,j}s^{\omega}.\label{aa2}\end{array}
\end{equation}
For $s=0$, \eqref{aa1}, or \eqref{aa2} yields $c_{0,j}=\pi_{j}\prod_{k=1}^{N}D_{B,k}(0).$ We still need $N\sum_{k=1}^{N}m_{k}$ unknown terms. These terms are obtained as follows: 

Substituting $t=t_{f,k}$, $f=1,\ldots,m_{k}$, $k=1,\ldots,N,$ in \eqref{aa2}, and substituting \eqref{aa1} into the integral (recall that $Re(t_{f,k}<0)$, $\forall f,k$) the resulting equation yields
\begin{equation}
   \begin{array}{l}
     \Phi_{A,j}(-t_{f,k})\sum_{i=1}^{N}p_{i,j}\prod_{\nu\neq i}^{N}D_{B,\nu}(t_{f,k})\int_{-\infty}^{0}\frac{\sum_{\omega=0}^{\sum_{k=1}^{N}m_{k}}c_{\omega,j}(t_{f,k}y)^{\omega}}{\prod_{k=1}^{N}D_{B,k}(t_{f,k}y)} P(V^{-}\in dy)\vspace{2mm}\\\qquad\qquad\qquad\qquad=\sum_{\omega=0}^{\sum_{k=1}^{N}m_{k}}c_{\omega,j}(t_{f,k})^{\omega},\,j=1,\ldots,N.
 \end{array}
 \label{sys11}
\end{equation}
The linear system of equations \eqref{sys11} will provide the unknowns $c_{\omega,j}$, $\omega=1,\ldots,\sum_{k=1}^{N}m_{k}$, $j=1,\ldots,N$, so that for $Re(s)\geq 0$,
\begin{equation}
    \Phi_{W,j}(s)=\frac{\sum_{\omega=0}^{\sum_{k=1}^{N}m_{k}}c_{\omega,j}s^{\omega}}{\prod_{k=1}^{N}D_{B,k}(s)},\,j=1,\ldots,N.\label{soaa1}
\end{equation}
%\subsection{Allowing additional dependence based on the FGM copula}
\section{Model II: On a Markov-modulated queue with service and interarrival times depending on waiting times} \label{mod2}
We now focus on Markov-modulated recursions of the form
\begin{equation}
    W_{n+1}=[V_{n}W_{n}+Y_{n}(V_{n})]^{+},\label{bqq}
\end{equation}
where $P(V_{n}=1)=p$, $P(V_{n}=-1)=q:=1-p$, and
\begin{displaymath}
    Y_{n}(V_{n})=\left\{\begin{array}{ll}
         S_{n}-A_{n+1},&\text{ when }V_{n}=1,  \\
         D_{n+1}-C_{n},& \text{ when }V_{n}=-1.
    \end{array}\right.
\end{displaymath}

We now assume that $A_{n}|Z_{n}=j\sim exp(\lambda_{j})$, $D_{n}|Z_{n}=j\sim exp(\mu_{j})$, while $S_{n}|Z_{n}=j$, $C_{n}|Z_{n}=j$, are arbitrarily distributed with distribution functions $B_{j}(.)$, and $C_{j}(.)$, $j=1,\ldots,N,$ respectively. Moreover, denote their LSTs with $\beta_{j}^{*}(s)=\int_{0}^{\infty}e^{-st}dB_{j}(t)$, $c_{j}^{*}(s)=\int_{0}^{\infty}e^{-st}dC_{j}(t)$. Assume also that $\gamma_{j}=\int_{0}^{\infty}xdB_{j}(x)$, $\delta_{j}=\int_{0}^{\infty}xdC_{j}(x)$, $j=1,\ldots,N$.

Our focus here is on the case where $p\in(0,1)$, so that the service times and the interarrival
times depend linearly and randomly on the waiting times, as well as on the value of $V_{n}$.
\subsection{On the stability condition}
The stability mechanism of our model is the Markov-modulated analogue of Whitt’s domination/regeneration argument \cite{whitt}. A similar procedure was also applied in the scalar case of our model in \cite{boxvla}.

We first recall the stability mechanism underlying this class of recursions. In the scalar non-modulated case, the recursion
\[
W_{n+1}=[C_nW_n+X_n]^+
\]
was studied in \cite{whitt}, in which the author used a comparison with an unrestricted affine recursion to obtain stochastic boundedness and regeneration conditions. The scalar recursion with $C_n\in\{1,-1\}$ was later analysed in~\cite{boxvla}, where the stability result in \cite{whitt} was recalled as the basis for the subsequent exact transform analysis. The stochastic boundedness of the dominating affine recursion is ultimately related to the stability theory of stochastic difference equations with stationary coefficients, as developed in \cite{brandt}.

Let $V_n^+=\max\{0,V_n\}$, $Y_n^+=\max\{0,Y_n(V_n)\}$. For $w\ge0$, $v\in\{1,-1\}$, and $y\in\mathbb R$, we have
\[
[vw+y]^+\le v^+w+y^+.
\]
Indeed, if $v=1$, then $[w+y]^+\le w+y^+$, while if $v=-1$, then $[-w+y]^+\le y^+$. Hence, if $W_n$ satisfies \eqref{bqq} and $Q_n$ is defined by
\[
Q_{n+1}=V_n^+Q_n+Y_n^+,\qquad Q_0=W_0,
\]
then $W_n\le Q_n$, $n\ge0$, almost surely. We use tightness in the standard sense: a family of nonnegative random variables $\{X_n\}_{n\in\mathbb{N}_{0}}$ is tight if
\[
\lim_{x\to\infty}\sup_{n\ge0}P(X_n>x)=0.
\]
Therefore, tightness of $\{Q_n\}$ implies tightness of $\{W_n\}$, since
\[
P(W_n>x)\le P(Q_n>x),\qquad x\ge0.
\]

When $p<1$, we have $P(V_n=-1)=q>0$, and hence $V_n^+=0$ with positive probability. At such epochs, the auxiliary recursion $Q_n$ loses dependence on its previous value. Under the finite first-moment assumptions
\[
\max_{i\in E}E(S_n\mid Z_n=i)<\infty,
\qquad
\max_{j\in E}E(D_{n+1}\mid Z_{n+1}=j)<\infty,
\]
and since the background chain is finite and irreducible, the auxiliary process is stochastically bounded. Consequently, $\{W_n\}$ is tight.

To obtain convergence to a proper limiting distribution, we use the regeneration mechanism of the original recursion. When $V_n=-1$,
\[
W_{n+1}=[-W_n+D_{n+1}-C_n]^+.
\]
Thus, on the event $\{V_n=-1,\ D_{n+1}\le C_n\}$, we have $W_{n+1}=0$, independently of the previous value of $W_n$. That is the regeneration event. Then, under stationarity of the background chain
\[
P(V_n=-1,\ D_{n+1}\le C_n)=q\sum_{i=1}^N\sum_{j=1}^N
\pi_i p_{i,j}
P(D_{n+1}\le C_n\mid Z_n=i,Z_{n+1}=j).
\]
Assuming that $P(V_n=-1,\ D_{n+1}\le C_n)>0$, the joint process $(W_n,Z_n)$ returns to the set $\{0\}\times E$ with positive probability. Since the background chain is finite and irreducible, returns to states of the form $(0,j)$, $j\in E$, form regeneration epochs. Under the above reset condition and finite first-moment assumptions, the regeneration cycles have finite mean length. Hence $(W_n,Z_n)$ admits a proper stationary distribution and $W_n\Rightarrow W$, as $n\to\infty$.

In the limiting case $p=1$, the reset mechanism disappears and the recursion reduces to the usual Markov-modulated Lindley-type recursion. In that case one needs the standard traffic condition for stability; see \cite{adan2}.
\subsection{Stationary analysis}
Recall that $\Phi_{W,i}^{n}(s)=E(e^{-sW_{n}}1_{\{Z_{n}=i\}})$, $Re(s)\geq 0$, $n\geq 0$, $i=1,\ldots,N$, and assuming the limit exists, let $\Phi_{W,i}(s)=\lim_{n\to\infty}\Phi_{W,i}^{n}(s)$. Some additional notation: $\Lambda=diag(\lambda_{1},\ldots,\lambda_{N})$, $M=diag(\mu_{1},\ldots,\mu_{N})$, and
\begin{displaymath}
    \begin{array}{rl}
         B^{*}(s)=&diag(\beta^{*}_{1}(s),\ldots,\beta_{N}^{*}(s)), \vspace{2mm} \\
         H(s)=&pB^{*}(s)P\Lambda.
    \end{array}
\end{displaymath}
\begin{theorem}\label{th1}
    The transform row vector $\tilde{\Phi}_{W}^{T}(s)$ satisfies
    \begin{equation}
    \begin{array}{rl}
       \tilde{\Phi}_{W}^{T}(s)[H(s)+sI-\Lambda]=&\tilde{v}(s),
       \end{array} \label{fun}
    \end{equation}
    where, $\tilde{\Phi}_{W}^{T}(0)=\tilde{\pi}^{T}$,
    \begin{equation}
        \begin{array}{rl}
             \tilde{v}(s)=&(s\tilde{k}_{1}+s^{2}\tilde{k}_{2}-q\tilde{\pi}^T\Lambda M)(M+sI)^{-1},  \vspace{2mm}\\
             \tilde{k}_{l}=&(k_{1,l},\ldots,k_{N,l}),\,l=1,2,\vspace{2mm}\\
             k_{j,1}=&q\pi_{j}\mu_{j}+\mu_{j}v_{j}^{(1)}-\lambda_{j}v_{j}^{(-1)},\,j=1,\ldots,N,\vspace{2mm}\\
            k_{j,2}=&v_{j}^{(1)}+v_{j}^{(-1)},\,j=1,\ldots,N,\vspace{2mm}\\
             v_{j}^{(1)}=&p\sum_{i=1}^{N}p_{i,j}\beta_{i}^{*}(\lambda_{j})\Phi_{W,i}(\lambda_{j}),\,j=1,\ldots,N,\vspace{2mm}\\
             v_{j}^{(-1)}=&q[\pi_{j}-\sum_{i=1}^{N}p_{i,j}c_{i}^{*}(\mu_{j})\Phi_{W,i}(\mu_{j})],\,j=1,\ldots,N,
        \end{array}\label{consist}
    \end{equation}
\end{theorem}
\begin{proof}
    The transform $\Phi_{W,j}^{n+1}(s):=E(e^{-sW_{n+1}}1_{\{Z_{n+1}=j\}})$ satisfies
    \begin{equation}
        \begin{array}{l}
            E(e^{-sW_{n+1}}1_{\{Z_{n+1}=j\}})\vspace{2mm}\\= pE(e^{-s[S_{n}-A_{n+1}+W_{n}]^{+}}1_{\{Z_{n+1}=j\}})+qE(e^{-s[D_{n+1}-C_{n}-W_{n}]^{+}}1_{\{Z_{n+1}=j\}}).
        \end{array}\label{eq1}
    \end{equation}
    Using standard arguments and following \cite{adan2,vlasioudep} we have
    \begin{displaymath}
        \begin{array}{rl}
            E(e^{-s[S_{n}-A_{n+1}+W_{n}]^{+}}1_{\{Z_{n+1}=j\}})= &\sum_{i=1}^{N}p_{i,j} [\frac{\lambda_{j}}{\lambda_{j}-s}\Phi_{W,i}^{n}(s)\beta_{i}^{*}(s)-\frac{s}{\lambda_{j}-s}\Phi_{W,i}^{n}(\lambda_{j})\beta_{i}^{*}(\lambda_{j})],\vspace{2mm} \\
             E(e^{-s[D_{n+1}-C_{n}-W_{n}]^{+}}1_{\{Z_{n+1}=j\}})=&\sum_{i=1}^{N}p_{i,j}[P(Z_{n}=i)-\frac{s}{\mu_{j}+s}\Phi_{W,i}^{n}(\mu_{j})c_{i}^{*}(\mu_{j})].
        \end{array}
    \end{displaymath}
    Substituting back in \eqref{eq1}, and letting $n\to\infty$ so that $\lim_{n\to\infty}\Phi_{W,i}^{n}(s)=\Phi_{W,i}(s)$, we obtain after some algebra
    \begin{equation}
        \begin{array}{l}
             \lambda_{j}p\sum_{i=1}^{N}p_{i,j} \Phi_{W,i}(s)\beta_{i}^{*}(s)-(\lambda_{j}-s)\Phi_{W,j}(s)=\frac{1}{\mu_{j}+s}\left[s^{2}k_{j,2}+sk_{j,1}-q\pi_{j}\lambda_{j}\mu_{j}\right].             
        \end{array}\label{q0}
    \end{equation}
       In matrix form equation \eqref{q0} yields \eqref{fun}. Moreover, $\Phi_{W,j}^{n+1}(0)=E(1_{\{Z_{n+1}=j\}})$, and as $n\to\infty$, $\Phi_{W,j}^{n}(0)\to\pi_{j}$, thus, $\tilde{\Phi}^{T}_{W}(0)=\tilde{\pi}^{T}$. This completes the proof.
\end{proof}

Note that in order to fully specify $\tilde{\Phi}^{T}_{W}(s)$ we first need to obtain $2N^{2}$ unknown terms, namely $\Phi_{W,i}(\lambda_{j})$, $\Phi_{W,i}(\mu_{j})$, $i,j=1,\ldots,N$, which are involved in the vectors $\tilde{k}_{1}$, $\tilde{k}_{2}$. This task is accomplished by performing the following steps:

\begin{enumerate}
    \item
First, the following theorem gives the number and placement of the solutions of $det(H(s)+sI-\Lambda)=0$. As shown below (see Theorem \ref{theor1}), that result will provide us with $N$ equations that will be helpful in determining the vectors $\tilde{k}_{1}$, $\tilde{k}_{2}$.
\begin{theorem}\label{theo}
    The equation
    \begin{equation}
    det(H(s)+sI-\Lambda)=0,
\label{eig}
    \end{equation}
    has exactly $N$ zeros, say $s_{1},\ldots,s_{N}$, with $Re(s_{i})>0$, $i=1,\ldots,N$.
\end{theorem}
\begin{proof}
   See Appendix \ref{roots}.
\end{proof}
Now we state the following theorem.
\begin{theorem}\label{theor1}
    Suppose that the $N$ solutions $s_i$, $i=1,2,\ldots,N,$ with $Re(s_i) > 0$ to equation \eqref{eig} are distinct. Let $\tilde{a}_i=(a_{i,1},\ldots,a_{i,N})^{T}$ be a nonzero
column vector satisfying
\begin{equation*}
    (H(s_{i})+s_{i}I-\Lambda)\tilde{a}_{i}=\tilde{0},\,i=1,\ldots,N,\label{eig2}
\end{equation*}
where $\tilde{0}$ is a row vector of zeros. Then, for $i=1,\ldots,N$, $\tilde{v}(s_{i})\tilde{a}_{i}=0$, or equivalently,
\begin{equation}
        \sum_{k=1}^{N}v_{k}(s_{i})a_{i,k}=0,
   \label{n1}
\end{equation}
where $\tilde{v}(s)=(v_{1}(s),\ldots,v_{N}(s))$ with
\begin{displaymath}
    v_{j}(s)=\frac{1}{\mu_{j}+s}[-q\pi_{j}\lambda_{j}\mu_{j}+sk_{j,1}+s^{2}k_{j,2}],\,j=1,\ldots,N.
\end{displaymath}
\end{theorem}
\begin{proof}
    Since $s_i$ satisfies equation \eqref{eig}, then there is a nonzero column vector $\tilde{a}_{i}$ such that 
    \begin{equation}
    (H(s_{i})+s_{i}I-\Lambda)\tilde{a}_{i}=\tilde{0},\,i=1,\ldots,N.\label{eig20}
\end{equation}
Post-multiplying \eqref{fun} for $s=s_{i}$ with $\tilde{a}_{i}$, we get \eqref{n1}, which provides $N$ linear equations that relate the elements of $\tilde{k}_{1}$, $\tilde{k}_{2}$.
\end{proof}
\item 
Substituting $s=\mu_{l}$, for each $l$, $l=1,\ldots,N$ in \eqref{q0}, we obtain $N^{2}$ equations that relate $\Phi_{W,i}(\lambda_{j})$, $\Phi_{W,i}(\mu_{j})$, $i,j=1,\ldots,N$. More precisely, substituting $s=\mu_{l}$, $l=1,\ldots,N$ in \eqref{q0} we have for $j,l=1,\ldots,N$:
\begin{equation}
    \begin{array}{l}
        \lambda_{j}p\sum_{i=1}^{N}p_{i,j} \Phi_{W,i}(\mu_{l})\beta_{i}^{*}(\mu_{l})-(\lambda_{j}-\mu_{l})\Phi_{W,j}(\mu_{l})\vspace{2mm}\\=\frac{1}{\mu_{j}+\mu_{l}}\left[\mu_{l}(\mu_{l}+\mu_{j})p\sum_{i=1}^{N}p_{i,j} \Phi_{W,i}(\lambda_{j})\beta_{i}^{*}(\lambda_{j})\right.\vspace{2mm}\\
        \left.+\mu_{l}q(\lambda_{j}-\mu_{l})\sum_{i=1}^{N}p_{i,j} \Phi_{W,i}(\mu_{j})c_{i}^{*}(\mu_{j})+(\mu_{l}+\mu_{j})q\pi_{j}(\mu_{l}-\lambda_{j})\right]. 
    \end{array}\label{q2}
\end{equation}
\item 
We also need the $N(N-1)$ equations. These are obtained from \eqref{q0} as follows: For each $j=1,\ldots,N$, substitute in \eqref{q0} $s=\lambda_{k}$, $k\neq j$, $k=1,\ldots,N$. Then, we have for $j=1,\ldots,N$, $k\neq j$:
\begin{equation}
    \begin{array}{l}
        \lambda_{j}p\sum_{i=1}^{N}p_{i,j} \Phi_{W,i}(\lambda_{k})\beta_{i}^{*}(\lambda_{k})-(\lambda_{j}-\lambda_{k})\Phi_{W,j}(\lambda_{k})\vspace{2mm}\\=\frac{1}{\mu_{j}+\lambda_{k}}\left[\lambda_{k}(\lambda_{k}+\mu_{j})p\sum_{i=1}^{N}p_{i,j} \Phi_{W,i}(\lambda_{j})\beta_{i}^{*}(\lambda_{j})\right.\vspace{2mm}\\
        \left.+\lambda_{k}q(\lambda_{j}-\lambda_{k})\sum_{i=1}^{N}p_{i,j} \Phi_{W,i}(\mu_{j})c_{i}^{*}(\mu_{j})+(\lambda_{k}+\mu_{j})\pi_{j}q(\lambda_k-\lambda_j)\right]. 
    \end{array}\label{q3}
\end{equation}
\end{enumerate}
Then, \eqref{n1}, \eqref{q2}, \eqref{q3} constitute $2N^{2}$ linear equations for the $2N^{2}$ unknowns $\Phi_{W,i}(\lambda_{j})$, $\Phi_{W,i}(\mu_{j})$, $i,j=1,\ldots,N$. Once $\Phi_{W,i}(\lambda_{j})$, $\Phi_{W,i}(\mu_{j})$, $i,j=1,\ldots,N$ are known, the vectors $\tilde{k}_{1}$, $\tilde{k}_{2}$ are known, and thus, $\tilde{v}(s)$ is known, which in turn implies that $\tilde{\Phi}^{T}_{W}(s)$ is known. 
\subsection{Steady-state moments}\label{mom}
Once $\tilde{k}_{1}$, $\tilde{k}_{2}$ are known, the entire transform vector $\tilde{\Phi}^{T}_{W}(s)$ is known, and ready to be used to compute
the moments of the steady-state waiting time distribution. They are given in the following theorem. Let
\begin{displaymath}
    \begin{array}{rl}
m_{r,i}=\lim_{n\to\infty}E(W_{n}^{r}1_{\{Z_{n}=i\}}),&r=0,1,2,\ldots,\,i=1,\ldots,N,  \vspace{2mm}\\
         \tilde{m}_{r}=(m_{r,1},\ldots,m_{r,N}),& r=0,1,2,\ldots,\vspace{2mm}\\
         \gamma_{r,i}=\int_{0}^{\infty}x^{r}dB_{i}(x),&r=0,1,2,\ldots,\,i=1,\ldots,N,\vspace{2mm}\\
         \Gamma_{r}=diag(\gamma_{r,1},\ldots,\gamma_{r,N}),& r=0,1,2,\ldots.\\
    \end{array}
\end{displaymath}
Note that $\Gamma_{0}=I$ and $\Gamma_{1}=\Gamma$. In particular,
\begin{displaymath}
    \tilde m_0=\tilde\pi^{T},
\end{displaymath}
since $\tilde\Phi_{W}^{T}(0)=\tilde{\pi}^T$. 

For each order $r$, the moment recursion given below in Theorem~\ref{th5} is
understood under the assumption that the stationary waiting-time moments
through order $r$ and the service-time moments through order $r$ are
finite. The expansions in the proof are used coefficientwise as
$s\downarrow0$, up to the order being considered. This procedure follows the approach in
\cite[Theorem~3.4, Remark~3.5]{adan2}.

\begin{theorem}\label{th5}
    The moment-vectors $\tilde{m}_{r}$ satisfy the following recursive equations: 
\begin{equation}
    \tilde{m}_r=(-1)^rr!\left[b_r-\sum_{j=0}^{r-1}\frac{(-1)^j}{j!}\tilde{m}_jK_{r-j}\right]K_0^{-1},\qquad r\geq 1,
    \label{moments}
\end{equation}
where,
\begin{displaymath}
    \begin{array}{rl}
       K_0&=(pP-I)\Lambda M,\\
K_1&=p(P\Lambda-\Gamma_1P\Lambda M)+M-\Lambda,\\
K_2&=p\left(\frac{1}{2}\Gamma_2P\Lambda M-\Gamma_1P\Lambda\right)+I,
    \end{array}
\end{displaymath}
for $r\geq 3$,
\begin{displaymath}
    K_r=p\left(
\frac{(-1)^{r-1}}{(r-1)!}\Gamma_{r-1}P\Lambda +\frac{(-1)^r}{r!}\Gamma_rP\Lambda M
\right),
\end{displaymath}
and
\begin{displaymath}
b_0:=-q\tilde{\pi}^{\,T}\Lambda M,\quad
b_1:=\tilde{k}_1,\quad
b_2:=\tilde{k}_2,\quad
b_r:=0,\quad r\geq3.
\end{displaymath}
Here
\[
K_0^{-1}=(\Lambda M)^{-1}(pP-I)^{-1},
\]
exists since $\Lambda$ and $M$ are invertible and
$\rho(pP)=p<1$ (where $\rho(A)$ denotes the spectral radius of the matrix
$A$).
\end{theorem}
\begin{proof}
    We have that $\tilde{\Phi}^{T}_{W}(s)=\sum_{r=0}^{\infty}a_r s^r$, where $a_r:=(-1)^{r}\frac{\tilde{m}_{r}}{r!}$, $r=0,1,\ldots$, $B^{*}(s)=\sum_{r=0}^{\infty}(-1)^{r}\Gamma_{r}\frac{s^{r}}{r!}$,
   % \begin{displaymath}
    %    \begin{array}{lr}
     %       \tilde{\Phi}^{T}_{W}(s)=\sum_{r=0}^{\infty}(-1)^{r}\tilde{m}_{r}\frac{s^{r}}{r!},&
      %       B^{*}(s)=\sum_{r=0}^{\infty}(-1)^{r}\Gamma_{r}\frac{s^{r}}{r!},
       % \end{array}
    %\end{displaymath}
    thus,
    \begin{displaymath}
        H(s)=p\sum_{r=0}^{\infty}(-1)^{r}\Gamma_{r}P\Lambda\frac{s^{r}}{r!}.
    \end{displaymath}
    Note that the displayed infinite sums are understood as formal coefficient expansions, in the sense that for a fixed order \(R\), only terms through \(s^R\) are used. 
    
    Let,
    \[
\mathcal K(s):=(H(s)+sI-\Lambda)(M+sI).
\]
   Then, \eqref{fun} is now written as 
    \begin{displaymath}
        \tilde{\Phi}_{W}^{T}(s)\mathcal K(s)=s\tilde{k}_{1}+s^{2}\tilde{k}_{2}-q\tilde{\pi}^T\Lambda M,
    \end{displaymath}
    or equivalently
    \begin{displaymath}
    \begin{array}{l}
        \sum_{r=0}^{\infty}(-1)^{r}\tilde{m}_{r}\frac{s^{r}}{r!}\left[p\sum_{r=0}^{\infty}(-1)^{r}\Gamma_{r}P\Lambda(M\frac{s^{r}}{r!}+\frac{s^{r+1}}{r!})+s(M-\Lambda)+s^{2}I-\Lambda M\right] \vspace{2mm}\\
        =s\tilde{k}_{1}+s^{2}\tilde{k}_{2}-q\tilde{\pi}^T\Lambda M.
    \end{array} 
    \end{displaymath}
    Clearly, \(K_r\), $r\geq 0$ are the coefficients in the expansion of \(\mathcal K(s)\). Coefficient matching gives
\begin{displaymath}
    \sum_{j=0}^r a_jK_{r-j}=b_r.
\end{displaymath}
At orders \(s\) and \(s^2\), and having in mind that $\tilde{m}_0=\tilde{\pi}^T$, this is
\[
\tilde{\pi}^T K_1-\tilde m_1K_0=\tilde k_1,
\qquad
\tilde{\pi}^T K_2-\tilde m_1K_1+\frac12\tilde m_2K_0=\tilde k_2,
\]
which yields exactly
\[
\tilde m_1=(\tilde{\pi}^T K_1-\tilde k_1)K_0^{-1},
\qquad
\tilde m_2=2(\tilde k_2+\tilde m_1K_1-\tilde{\pi}^T K_2)K_0^{-1},
\]
and where $K_0^{-1}$ exists since $\Lambda$ and $M$ are invertible and the spectral radius of $pP$ satisfies $\rho(pP)=p<1$. Proceeding similarly, equating coefficients of $s^{r}$, $r\geq 3$ we get the remaining moments. It is easily seen that the moments satisfy the recursive expression \eqref{moments}.
\end{proof}

\subsection{Asymptotic behaviour: the light-tailed case}\label{asym}

Let $G(s):=H(s)+sI-\Lambda$, and let $Adj(G(s))$ denote the adjugate matrix of $G(s)$. From Theorem~\ref{th1}, we have
\begin{equation}
\tilde{\Phi}_{W}^{T}(s)=\tilde v(s)
\frac{Adj(G(s))}{det(G(s))}.
\label{we1}
\end{equation}
We restrict our attention to the light-tailed case. More precisely, we assume that the transforms $\beta_i^*(s)$, $i=1,\ldots,N$, admit analytic continuation to a half-plane $Re(s)>-\eta$ for some $\eta>0$. We also assume that the remaining factors that appear in $\tilde v(s)$ and $G(s)$ are analytic in the region under consideration, except possibly in isolated singularities. Then the possible singularities of $\tilde{\Phi}_{W}^{T}(s)$ in this half-plane are determined by the zeros of $det(G(s))$, except when such zeros are canceled by the numerator.

Assume that $-R<0$ is the zero of $det(G(s))$ with the largest real part in the negative half-plane, that this zero is simple, and that it is not canceled by the numerator $\tilde {v}(s) Adj(G(s))$. Assume also that all other singularities have real part strictly smaller than $-R$. Then, $-R$ is the dominant singularity of the stationary transform.

For $j=1,\ldots,N$, define the defective stationary tail
\begin{displaymath}
    \overline F_{W,j}(x):=P(W>x,Z=j).
\end{displaymath}
Since
\begin{displaymath}
    \int_0^\infty e^{-sx}\overline F_{W,j}(x)dx=\frac{\Phi_{W,j}(0)-\Phi_{W,j}(s)}{s},
\quad Re(s)>0,
\end{displaymath}
the dominant singularity of the tail transform is also located at $s=-R$. By the classical Heaviside principle for light-tailed transform inversion, we obtain
\begin{displaymath}
    \overline F_{W,j}(x)\sim C_j e^{-Rx},
\qquad x\to\infty.
\end{displaymath}
The constants $C_j$ are obtained from the residues at $s=-R$. If $A_{j}$ is the residue of the $j$-th stationary transform $\Phi_{W,j}(s)$ at the dominant pole $s=-R$, i.e.,
\begin{displaymath}
    A_j=\lim_{s\to-R}(s+R)\Phi_{W,j}(s),
\end{displaymath}
then $C_j=\frac{A_j}{R}$. In vector form,
\begin{displaymath}
    \widetilde A^T= \frac{
\tilde v(-R)Adj(G(-R))
}{\left.\frac{d}{ds}det(G(s))\right|_{s=-R}},
\quad
\widetilde C^T=\frac{1}{R}\widetilde A^T .
\end{displaymath}
Thus, under the above light-tail assumptions, non-cancellation, and dominant simple-pole assumptions, the stationary tail probabilities decay exponentially at rate $R$:
\begin{displaymath}
    \overline F_{W,j}(x)\sim C_j e^{-Rx},
\qquad j=1,\ldots,N,\qquad x\to\infty.
\end{displaymath}
For the scalar random-sign Lindley recursion, related tail asymptotics are
derived in \cite{vlasiou2010tail}.
\begin{remark}
    The analysis we performed so far can be applied even in the case where, given $Z_{n}=i$, $Z_{n+1}=j$, the random variables $\{A_{n}\}_{n\in\mathbb{N}}$, $\{S_{n}\}_{n\in\mathbb{N_{0}}}$ (resp. $\{D_{n}\}_{n\in\mathbb{N}}$, $\{C_{n}\}_{n\in\mathbb{N_{0}}}$) are linearly dependent, i.e., $A_{n+1}:=c_{j}S_{n}+J_{n+1}$, for some constant $c_{j}\in(0,1)$, $j\in E$, where $J_{n+1}$ is conditionally independent of $S_{n}$ given the state of the background Markov chain (resp. $D_{n+1}:=d_{j}C_{n}+L_{n+1}$, for some constant $d_{i,j}\in(0,1)$, $i,j\in E$, where $L_{n+1}$ is conditionally independent of $C_{n}$, given the state of the background Markov chain). Indeed, without loss of generality, let $J_{n}\sim exp(\delta_{i})$ (resp. $L_{n}\sim exp(\nu_{i})$), given $Z_{n}=i$. Then, the transform $\Phi_{W,j}^{n+1}(s):=E(e^{-sW_{n+1}}1_{\{Z_{n+1}=j\}})$ satisfies
    \begin{equation}
        \begin{array}{l}
            E(e^{-sW_{n+1}}1_{\{Z_{n+1}=j\}})\vspace{2mm}\\= pE(e^{-s[S_{n}-A_{n+1}+W_{n}]^{+}}1_{\{Z_{n+1}=j\}})+qE(e^{-s[D_{n+1}-C_{n}-W_{n}]^{+}}1_{\{Z_{n+1}=j\}}),
        \end{array}\label{eqx1}
    \end{equation}
    where
    \begin{displaymath}
        \begin{array}{l}
            E(e^{-s[S_{n}-A_{n+1}+W_{n}]^{+}}1_{\{Z_{n+1}=j\}})\vspace{2mm}\\\qquad=\sum_{i=1}^{N}P(Z_{n}=i)p_{i,j}\left(E(\int_{0}^{W_{n}+\bar{c}_{j}S_{n}}e^{-s(W_{n}+\bar{c}_{j}S_{n}-y)}\delta_{j}e^{-\delta_{j}ydy}|Z_{n}=i)\right.\vspace{2mm}\\
            \qquad\left.+E(\int_{W_{n}+\bar{c}_{j}S_{n}}^{\infty}\delta_{j}e^{-\delta_{j}ydy}|Z_{n}=i)\right)\vspace{2mm} \\
            \qquad= \sum_{i=1}^{N}p_{i,j} [\frac{\delta_{j}}{\delta_{j}-s}\Phi_{W,i}^{n}(s)\beta_{i}^{*}(\bar{c}_{j}s)-\frac{s}{\delta_{j}-s}\Phi_{W,i}^{n}(\delta_{j})\beta_{i}^{*}(\bar{c}_{j}\delta_{j})],
            \end{array}
\end{displaymath}
and
\begin{displaymath}
    \begin{array}{l}
         E(e^{-s[D_{n+1}-C_{n}-W_{n}]^{+}}1_{\{Z_{n+1}=j\}})=\sum_{i=1}^{N}P(Z_{n}=i)p_{i,j}\left(E(\int_{0}^{W_{n}+\bar{d}_{j}C_{n}}\nu_{j}e^{-\nu_{j}ydy}|Z_{n}=i)\right.\vspace{2mm}\\
            \qquad\left.+E(\int_{W_{n}+\bar{d}_{j}C_{n}}^{\infty}e^{-s(W_{n}+\bar{d}_{j}C_{n}-y)}\nu_{j}e^{-\nu_{j}ydy}|Z_{n}=i)\right)\vspace{2mm}\\\qquad=\sum_{i=1}^{N}p_{i,j}[P(Z_{n}=i)-\frac{s}{\nu_{j}+s}\Phi_{W,i}^{n}(\nu_{j})c_{i}^{*}(\bar{d}_{j}\nu_{j})],
        \end{array}
    \end{displaymath}
    where $\bar{c}_{j}:=1-c_{j}$, $\bar{d}_{j}:=1-d_{j}$, $j\in E$. Substituting back in \eqref{eqx1} we come up with a vector-valued functional equation similar to \eqref{fun} that can be solved by using similar arguments as above. 

    In Appendix \ref{secfgm}, we go one step further and consider a mathematically interesting generalization, under which given $Z_{n}$, $Z_{n+1}$, the components of $Y_{n}(V_{n})$, i.e., $S_{n}$, $A_{n+1}$, and $D_{n+1}$, $C_{n}$, are dependent based on a special class of copulas \cite{nelsen}. 
\end{remark}

\section{Numerical results}\label{num}
\subsection{Example 1: Baseline transform computation and numerical validation}
We focus on Model II and consider a two-state background Markov chain with transition matrix
\[
P=
\begin{pmatrix}
0.2 & 0.8\\
0.6 & 0.4
\end{pmatrix},
\]
whose stationary distribution is denoted by \(\pi\). We take
\[
\Lambda=diag(\lambda_{1},\lambda_{2})=diag(2,3),
\qquad
M=diag(\mu_{1},\mu_{2})=diag(10/u,8/u),
\]
and assume exponential distributions for the random variables appearing in Model~II. In particular, \(A|Z=j\) has rate \(\lambda_j\), \(D|Z=j\) has rate \(\mu_j\), \(S|Z=i\) has rate \(\mu_i\), and \(C|Z=i\) has rates \(\theta_1=8,\theta_2=6\). Unless otherwise stated, we use \(u=2\) and \(p=0.6\), so that \(q=1-p=0.4\).
\subsubsection{Moment computation based on Theorem \ref{th5}}

We first validate the recursive moment formulas of Theorem \ref{th5}. This provides a numerically stable way to compute stationary performance measures once the constants \(\tilde k_1,\tilde k_2\) have been identified. For the baseline parameter set, the positive zeros of
\[
\det(H(s)+sI-\Lambda)=0
\]
are
\begin{equation}
    s_1=1.2770,\qquad s_2=2.9325.\label{positive}
\end{equation}

The first step in the numerical implementation is the identification of the constants $\tilde k_1$, $\tilde k_2$ appearing in Theorem~\ref{th1}. The theoretical characterization is given by the finite-dimensional system derived in Theorems~\ref{th1} and~\ref{theor1}. More precisely, equation~\eqref{n1} gives the root-cancellation equations, while equations~\eqref{q2} and~\eqref{q3} give the additional boundary-value equations obtained by evaluating the functional equation at $s=\mu_l$ and $s=\lambda_k$, $k\neq j$. Together, these equations determine the boundary values $\Phi_{W,i}(\lambda_j)$, $\Phi_{W,i}(\mu_j)$, $i,j=1,\ldots,N$, and hence determine $\tilde k_1$ and $\tilde k_2$ through their definitions in Theorem \ref{th1}.

For the numerical computations we use an equivalent reduced formulation in which the unknowns are directly the $2N$ constants of $\tilde k_1$, $\tilde k_2$. Indeed, for any $s$ such that $G(s)=H(s)+sI-\Lambda$ is nonsingular, Theorem~\ref{th1} gives
\begin{displaymath}
    \tilde\Phi_W^T(s)= \widetilde v(s)G(s)^{-1}.
\end{displaymath}
Therefore, the boundary values $\Phi_{W,i}(\lambda_j)$, $\Phi_{W,i}(\mu_j)$ can be expressed in terms of $\tilde k_1$, $\tilde k_2$, provided that $G(\lambda_j)$ and $G(\mu_j)$ are nonsingular. This is the case for the parameter values considered below. By substituting the boundary values in the definitions of $k_{j,1}$ and $k_{j,2}$ in \eqref{consist} Theorem~\ref{th1}, we have a linear system in the $2N$ unknowns of $\tilde k_1$, $\tilde k_2$.

Thus, the reduced system consists of the root-cancellation equations $\tilde v(s_r)\tilde a_r=0$, $r=1,\ldots,N$ in \eqref{n1} together with the equations obtained from the definitions of $k_{j,1}$, $k_{j,2}$ as described above. This avoids solving explicitly for all $2N^2$ boundary transform values, while giving the same constants. In the baseline example, the positive zeros of $\det G(s)=0$ are given in \eqref{positive}. The reduced identification system for the $2N$ constants contained in $\tilde k_1$, $\tilde k_2$ has full rank. For $N=2$, the numerical system has rank $4$, condition number $1.033\times 10^1$, and residual norm $6.894\times 10^{-16}$. The resulting constants are
\begin{equation}
    \tilde k_1=(1.3535,1.2015),
\qquad
\tilde k_2=(0.2452,0.2806).\label{cons}
\end{equation}
The small residual confirms that the cancellation and self-consistency equations are solved accurately.

Using these constants in the moment recursion of Theorem \ref{th5} gives the following stationary moment vectors:
\begin{table}[ht!]
    \centering
    \begin{tabular}{c|c|c}
    \hline
        $r$ & $E(W^r)$ & $\left(E(W^r1_{\{Z=1\}}),E(W^r1_{\{Z=2\}})\right)$\\
\hline
0 & $1.000000$ & $(0.4286,0.5714)$\\
1 & $0.148210$ & $(0.0568,0.0914)$\\
2 & $0.093571$ & $(0.0361,0.0575)$\\
3 & $0.089197$ & $(0.0348,0.0544)$\\
4 & $0.113836$ & $(0.0450,0.0688)$\\ \hline
    \end{tabular}
    \caption{Moments}
    \label{tab:placeholder}
\end{table}

The first two moments are $E(W)=0.148210$, $E(W^2)=0.093571$, and therefore
\[
\operatorname{Var}(W)=0.071605,
\qquad
SD(W)=0.267591,
\qquad
CV(W)=1.80549.
\]

To validate the recursion, we compare these quantities with Monte Carlo simulation based on \(10^6\) observations after burn-in. The results are reported in Table \ref{meancomp}.
\begin{table}[ht!]
    \centering
    \begin{tabular}{c|c|c|c}\hline
       \text{Quantity}
& \text{Theorem \ref{th5}}
& \text{Simulation}
& \text{Relative error}
\\
\hline
$E(W)$ & $0.148210$ & $0.148114$ & $0.00065$\\
$E(W^2)$ & $0.093571$ & $0.093989$ & $0.00444$\\
$Var(W)$ & $0.071605$ & $0.072051$ & $0.00619$\\
$SD(W)$ & $0.267591$ & $0.268420$ & $0.00310$\\
$CV(W)$ & $1.80549$ & $1.81230$ & $0.00375$\\
    \end{tabular}
    \caption{Moment comparison.}
    \label{meancomp}
\end{table}

The agreement is very close. In particular, the relative error is below \(0.1\%\) for the mean and below \(1\%\) for the second moment, variance, standard deviation and coefficient of variation. This confirms that the moment recursion gives accurate stationary performance measures. Higher-order moments can also be computed recursively. For example, Theorem \ref{th5} gives
\[
E(W^3)=0.089197,
\quad
E(W^4)=0.113836,
\quad
E(W^5)=0.182038,
\quad
E(W^6)=0.349770.
\]
The comparison with simulation becomes less accurate for higher moments, as expected, because high-order moments are increasingly sensitive to rare large values of \(W\). Thus, the increasing relative discrepancy for \(r=5,6\) mainly reflects the difficulty of estimating high-order moments by direct simulation. This illustrates an advantage of the recursive approach: once the transform constants have been identified, higher moments can be obtained directly without relying on rare-event sampling.

\subsubsection{Numerical inversion and distributional performance measures}

Having validated the moment recursion, we next use the same identified constants $\tilde k_1$, $\tilde k_2$ to evaluate the stationary transform given by Theorem \ref{th1}. This allows us to go beyond moments and compute the stationary distribution function, tail probabilities and quantiles by numerical Laplace inversion, as shown below.

\paragraph{Computation of the stationary transform}

%The stationary transform vector satisfies \eqref{fun} as shown in Theorem \ref{th1} where the vector \(\tilde v(s)\) depends on the two vectors of constants $\tilde{k}_1$, $\tilde{k}_{2}$. As shown in Theorem \ref{th1} these vectors of constants depend on $\Phi_{W,i}(\lambda_j)$, $\Phi_{W,i}(\mu_j)$, $i,j=1,\ldots,N$. However, the numerical implementation does not necessarily require solving directly for all these \(2N^2\) quantities. The transform depends on them only through the \(2N\) constants \(\tilde k_1,\tilde k_2\). We therefore identify these constants directly.
The stationary transform vector satisfies
\begin{displaymath}
    \tilde\Phi_W^T(s)G(s)=\widetilde v(s),
\qquad
G(s)=H(s)+sI-\Lambda.
\end{displaymath}
After the constants $\tilde k_1$, $\tilde k_2$ have been identified, the transform is evaluated at all nonsingular points by
\begin{displaymath}
    \tilde\Phi_W^T(s)=\widetilde v(s)G(s)^{-1}.
\end{displaymath}
Using the baseline values given above, the identified constants are those reported in \eqref{cons}. As a first validation of the transform computation, we check the normalization of the unconditional LST
\begin{displaymath}
    \Phi_W(s)=\mathbf 1^T\widetilde\Phi_W(s).
\end{displaymath}
Numerically, $\Phi_W(0.001)=0.999851836625$, which is consistent with $\Phi_W(0)=1$. The mean obtained from the derivative of the transform is $E(W_{LST})=0.14821010$, in close agreement with the simulation estimate $E(W_{sim})=0.14811397$.

\paragraph{Numerical inversion of the stationary transform}

The main advantage of the transform representation is that it provides access to the full stationary distribution, not only to moments. For the numerical inversion we evaluate the transform at the Fourier inversion points
\[
s_k=\frac{A+2\pi i k}{2x},\qquad k=0,1,\ldots .
\]
At such points, whenever \(\det G(s_k)\neq0\), we compute
\[
\tilde\Phi_W^T(s_k)=\tilde v(s_k)G(s_k)^{-1}.
\]
The positive zeros of \(\det G(s)\) are not poles of the stationary transform. They correspond to removable singularities, because the constants entering \(\tilde v(s)\) satisfy the root-cancellation equations. Nevertheless, direct numerical evaluation of \(G(s)^{-1}\) near such points may be ill-conditioned. Therefore, in the inversion experiments below, the inversion grid is chosen so that the evaluation points are sufficiently far from the roots of \(\det G(s)\). If an inversion point lies too close to such a root, the transform value is evaluated by a small complex perturbation, which approximates the removable analytic continuation. Once \(\tilde\Phi_W(s)\) has been computed, we obtain the unconditional LST
\[
\Phi_W(s)=\mathbf 1^T\tilde\Phi_W(s).
\]
The stationary distribution function and the stationary tail probability can then be recovered from
\begin{displaymath}
    F_W(x)= \mathcal L^{-1}
\left\{
\frac{\Phi_W(s)}{s}
\right\}(x),\qquad \bar F_W(x)=
P(W>x)
=
\mathcal L^{-1}
\left\{
\frac{1-\Phi_W(s)}{s}
\right\}(x).
\end{displaymath}
We use a Fourier-series numerical inversion method with Euler acceleration, in the spirit of Abate--Whitt inversion \cite{abate}.

Figure \ref{stat} compares the stationary distribution function obtained by Laplace inversion with the empirical distribution obtained by Monte Carlo simulation. The agreement is excellent over the entire plotted range. The value of the distribution function near the origin also reveals a substantial atom at zero, which is consistent with the reflective structure of the recursion.
\begin{figure}
    \centering
    \includegraphics[width=0.8\linewidth]{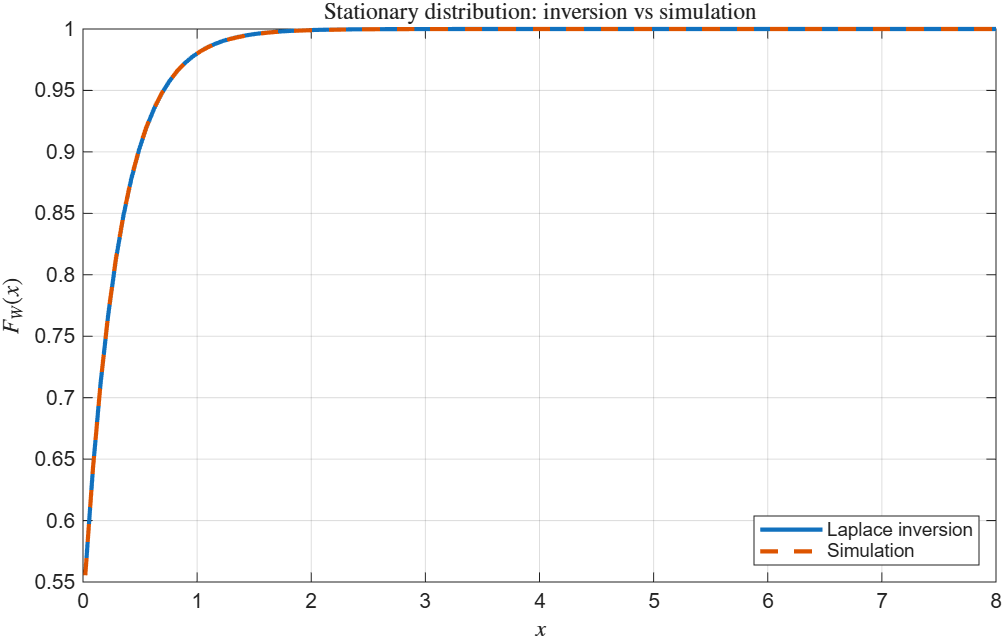}
    \caption{Stationary distribution of the waiting time.}
    \label{stat}
\end{figure}

Figure \ref{tails} compares the stationary tail probability \(P(W>x)\) obtained by transform inversion with the corresponding simulation estimate. Again, the agreement is very good for moderate tail probabilities. In the far tail, the simulation curve becomes irregular and eventually nearly flat. This is expected: with \(10^6\) simulated observations, probabilities of order \(10^{-6}\) or smaller are estimated from only a few observations. By contrast, the transform-based inversion remains numerically stable over the reported range. This illustrates a practical advantage of the transform approach: once the stationary transform has been computed, tail probabilities can be evaluated without requiring very long simulation runs.
\begin{figure}
    \centering
    \includegraphics[width=0.8\linewidth]{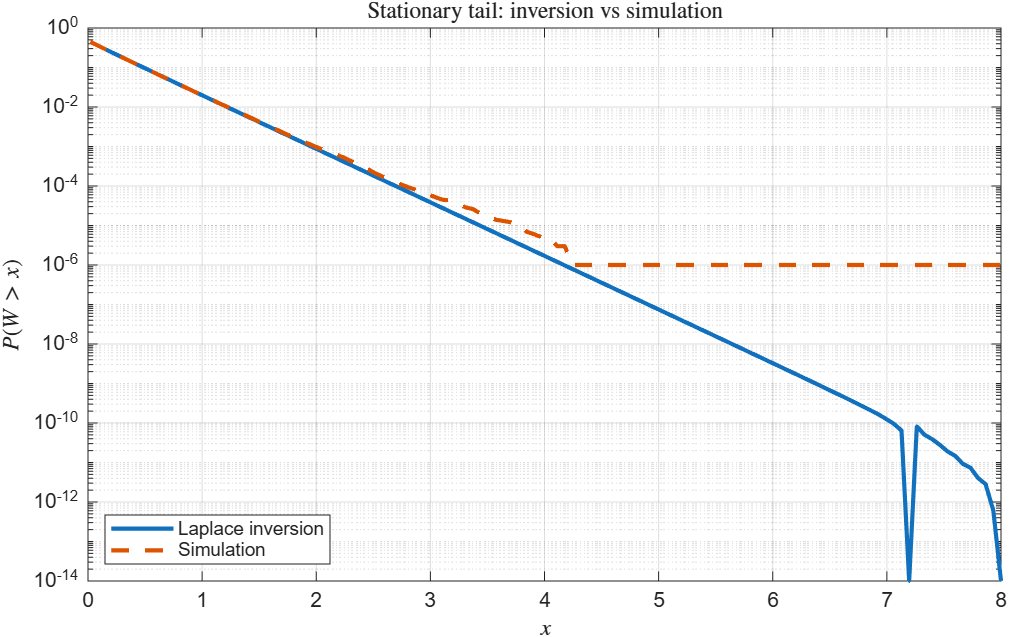}
    \caption{Stationary tail distribution of the waiting time.}
    \label{tails}
\end{figure}
The following table reports the numerical values of the stationary tail probabilities.
\begin{table}[ht!]
    \centering
    \begin{tabular}{c|c|c|c}
$x$ & $P(W>x)$\ \text{by inversion}
& $P(W>x)$\ \text{by simulation}
& \text{relative error}
\\ \hline
0.25 & $0.21190$ & $0.21122$ & $0.0032$\\
0.50 & $0.095591$ & $0.095149$ & $0.0046$\\
1.00 & $0.019805$ & $0.019861$ & $0.0028$\\
1.50 & $0.004151$ & $0.004241$ & $0.0212$\\
2.00 & $8.7316\times 10^{-4}$ & $9.7500\times 10^{-4}$ & $0.1044$\\
3.00 & $3.8613\times 10^{-5}$ & $5.8000\times 10^{-5}$ & $0.3343$\\
4.00 & $1.7011\times 10^{-6}$ & $5.0000\times 10^{-6}$ & $0.6598$\\
5.00 & $7.4737\times 10^{-8}$ & $0$ & $-$ \\ \hline
\end{tabular}
    \caption{Stationary tail probabilities: comparison.}
    \label{tail}
\end{table}

The increasing relative discrepancy in the far tail should not be interpreted as a failure of the transform method. Rather, it reflects the statistical limitation of the simulation estimator in a region where only a very small number of exceedances are observed. For example, a tail probability of order \(10^{-6}\) corresponds to approximately one observation in a simulation run of size \(10^6\). Hence relative errors in that range are unavoidable unless the simulation run is made much longer.

\paragraph{Quantiles}

We next compute stationary quantiles. This is another distributional performance measure that cannot be obtained from the mean alone. Let \(x_\alpha\) be defined by
\[
F_W(x_\alpha)=\alpha.
\]
Table \ref{quan} compares inversion-based and simulation-based quantiles. The agreement is very close. In particular, even the \(0.99\)-quantile is recovered accurately by numerical inversion. This confirms that the transform representation can be used to compute practically relevant performance measures beyond the stationary mean.
\begin{table}[ht!]
    \centering
    \begin{tabular}{c|c|c|c}
       $\alpha$
& $x_\alpha$\ \text{by inversion}
&$ x_\alpha$\ \text{by simulation}
& \text{absolute error}
\\ \hline
0.90 & 0.48587 & 0.48411 & 0.00176\\
0.95 & 0.70521 & 0.70414 & 0.00107\\
0.99 & 1.21850 & 1.22290 & 0.00440\\
    \end{tabular}
    \caption{Quantiles.}
    \label{quan}
\end{table}
\paragraph{Sensitivity with respect to \(p\) and \(u\)}

We now investigate how $E(W)$, $Var(W)$, $CV(W)$ changes as a function of the probability \(p=P(V=1)\) and the scaling parameter \(u\). Recall that \(p\) controls the frequency with which the recursion behaves like the ordinary Lindley-type recursion, while \(q=1-p\) controls the frequency of the \(V=-1\) mechanism. The parameter \(u\) scales the service-related rates through
\[
M=diag(10/u,8/u).
\]
Thus, larger values of \(u\) correspond to larger service-time scales. The remaining baseline values for the arrival rates are kept fixed: $\Lambda=diag(2,3)$.

Figure \ref{meanw} displays \(E(W)\) as a function of \(p\) for several values of \(u\). For fixed \(u\), the mean waiting time increases as \(p\) approaches one. This is natural: when \(p\) is close to one, the stabilizing effect of the \(V=-1\) mechanism occurs rarely, and the process behaves more like a conventional Markov-modulated Lindley recursion. The increase is much sharper for larger values of \(u\), because larger \(u\) corresponds to larger service times. Hence, the combined effect of high \(p\) and large \(u\) moves the system toward a high-workload regime.
\begin{figure}
    \centering
    \includegraphics[width=0.8\linewidth]{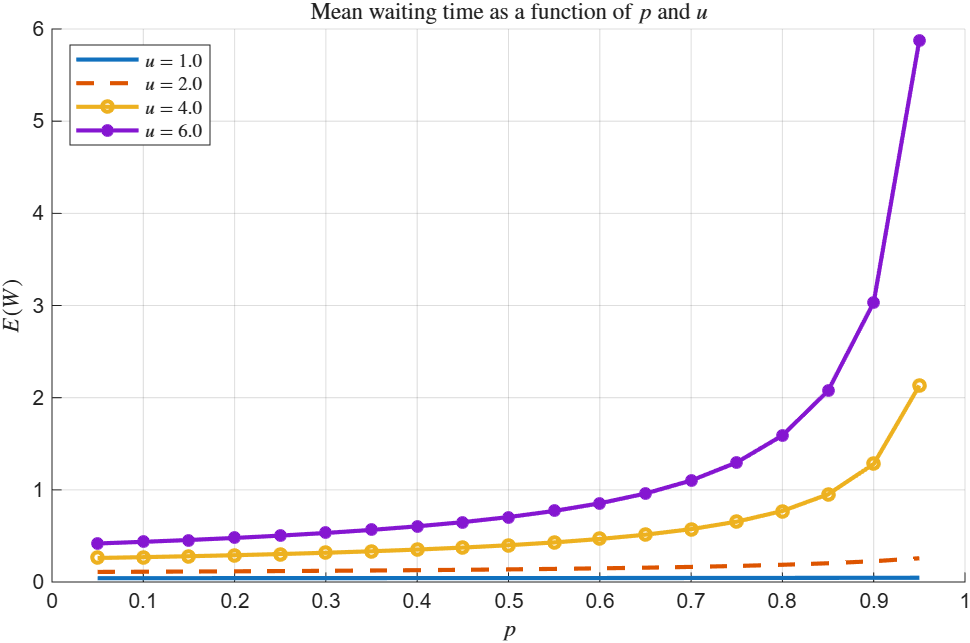}
    \caption{$E(W)$ by using Theorem \ref{th5}.}
    \label{meanw}
\end{figure}
This sensitivity experiment complements the inversion results. The mean plot illustrates the effect of the model parameters, while the inversion results show that the transform method provides access to the full stationary distribution, including tails and quantiles.

We also use Theorem \ref{th5} to compute $Var(W)$ and $CV(W)$ over the same parameter grid. As $p$ increases, the stabilizing $V=-1$ mechanism occurs less frequently, and the process behaves increasingly like a conventional Markov-modulated Lindley recursion. As $u$ increases, the service-time scale becomes larger. The combined effect of large $p$ and large $u$ therefore leads not only to a larger mean workload, but to larger variance. However, CV may decrease because the mean grows faster than the standard deviation.

More precisely, Figure \ref{var} shows that $Var(W)$ increases with $p$ and $u$. 
The behaviour of $CV(W)$, shown in Figure \ref{cv} is more nuanced. Since $CV(W)=\sqrt{Var(W)}/E(W)$ measures relative rather than absolute variability, $Var(W)$ increases, $CV(W)$ may decrease for large values of $p$ and $u$, because the mean workload increases faster than the standard deviation. Thus, in the high-workload regime, $E(W)$, and $Var(W)$ increase, however, the CV may decrease. Thus, although the workload becomes larger and more variable in absolute units, its fluctuations become smaller relative to its average level.
\begin{figure}
    \centering
    \includegraphics[width=0.8\linewidth]{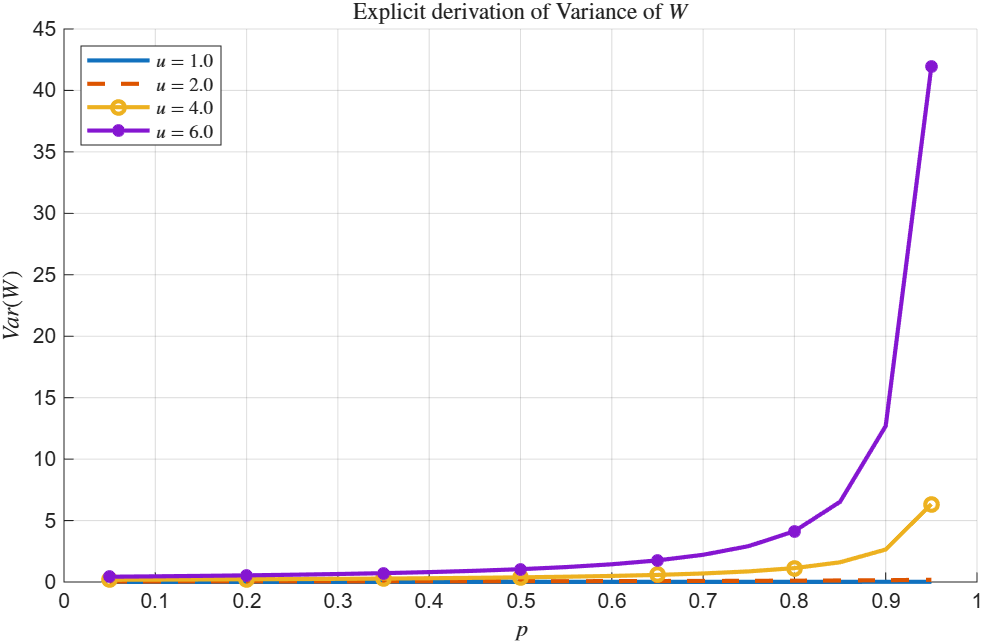}
    \caption{$Var(W)$ by using Theorem \ref{th5}.}
    \label{var}
\end{figure}
\begin{figure}
    \centering
    \includegraphics[width=0.8\linewidth]{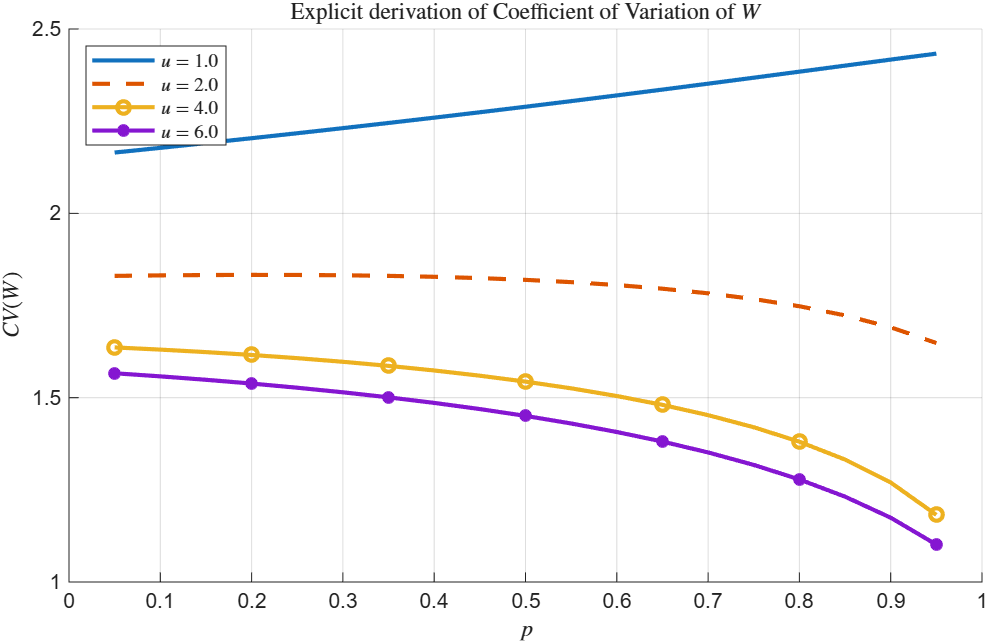}
    \caption{$CV(W)$ by using Theorem \ref{th5}.}
    \label{cv}
\end{figure}
%\subsubsection{Conclusion of the numerical study}

%The numerical experiments support three conclusions. First, the transform representation is numerically stable: the root-cancellation system is well-conditioned in the baseline example, the transform is correctly normalized, and the transform-based mean agrees with simulation. Second, numerical Laplace inversion accurately recovers the stationary distribution and quantiles. Third, the transform method is especially useful for tail probabilities, where direct simulation becomes increasingly noisy and computationally expensive.

\paragraph{The effect of Markov-modulation}

We finally use the transform-based moment recursion to isolate the effect of Markov modulation. The purpose of this experiment is to vary the temporal dependence of the background environment while keeping its stationary distribution fixed. In this way, the marginal mixture distributions of the primitive random variables remain unchanged, and any observed change in the workload distribution is due to the correlation structure induced by the Markov chain.

Throughout this subsection, the primitive distributions are kept fixed at their baseline values. We take
\[
p=0.6,\qquad q=0.4,\qquad
\Lambda=diag(2,3),\qquad M=diag(5,4),
\]
corresponding to \(u=2\). Thus,
\[
A_{n+1}\mid Z_{n+1}=j\sim exp(\lambda_j),
\qquad
D_{n+1}\mid Z_{n+1}=j\sim exp(\mu_j),
\]
while
\[
S_n\mid Z_n=i\sim exp(\mu_i),
\qquad
C_n\mid Z_n=i\sim exp(\eta_i),
\quad \eta=(\eta_1,\eta_{2})=(8,6).
\]
Only the transition matrix of the background chain is varied. In this way, the stationary distribution of the environment and the marginal
mixture distributions of the primitive random variables remain unchanged, and only the
temporal dependence induced by the Markov chain changes. We consider a two-state
background chain with fixed stationary distribution
\[
\tilde \pi^{T}=(\pi_{1},\pi_{2})=\left(\frac{3}{7},\frac{4}{7}\right).
\]
For \(\theta\in[-0.7,0.9]\), we use the transition matrix
\[
P(\theta)=
\begin{pmatrix}
\pi_1+\theta\pi_2 & \pi_2(1-\theta)\\
\pi_1(1-\theta) & \pi_2+\theta\pi_1
\end{pmatrix}.
\]
This parametrization keeps \(\tilde \pi\) fixed for all values of \(\theta\) considered.
Moreover, for the two-state chain,
\[
Corr(
\mathbf 1_{\{Z_n=i\}},
\mathbf 1_{\{Z_{n+1}=i\}})=\theta,\quad i=1,2.
\]
Thus, \(\theta\) measures the one-step persistence of the background environment:
positive values of \(\theta\) correspond to persistent environmental states, whereas
negative values correspond to a tendency to alternate between states.

Table \ref{tab:markov-impact} reports the induced correlations between the primitive random variables and the resulting stationary performance measures. The moments are computed from Theorem \ref{th5}, while the \(0.95\)-quantile is obtained from simulation. The results show that the temporal dependence of the environment has a visible impact on the stationary workload distribution, even though the stationary distribution of the background chain is kept fixed. As \(\theta\) increases from \(-0.7\) to \(0.9\), the stationary mean increases from \(0.14775\) to \(0.15183\), while the variance increases from \(0.07112\) to \(0.07656\). The coefficient of variation also increases, from \(1.8050\) to \(1.8224\), and the simulated \(0.95\)-quantile increases from \(0.70358\) to \(0.72277\).
\begin{table}[h!]
\centering
\caption{Impact of Markov modulation with fixed stationary distribution.}
\label{tab:markov-impact}
\begin{tabular}{c|cc|cccc}
\hline
\(\theta\) &
\(Corr(S_n,A_{n+1})\) &
\(Corr(C_n,D_{n+1})\) &
\(E(W)\) &
\(Var(W)\) &
\(CV(W)\) &
\(q_{0.95}\)\\
\hline
-0.7 & 0.01467  & -0.01019 & 0.14775 & 0.07112 & 1.8050 & 0.70358\\
-0.4 & 0.00838  & -0.00582 & 0.14821 & 0.07161 & 1.8055 & 0.70651\\
0    & 0.00000  & 0.00000  & 0.14896 & 0.07247 & 1.8072 & 0.70930\\
0.4  & -0.00838 & 0.00582  & 0.14995 & 0.07375 & 1.8110 & 0.71473\\
0.7  & -0.01467 & 0.01019  & 0.15096 & 0.07519 & 1.8165 & 0.71548\\
0.9  & -0.01886 & 0.01310  & 0.15183 & 0.07656 & 1.8224 & 0.72277\\
\hline
\end{tabular}
\end{table}

This confirms that Markov modulation affects not only the average workload but also the variability and high-quantile behavior of the stationary workload distribution. In particular, positive persistence in the background environment creates longer runs of similar operating conditions, which leads to larger workload fluctuations and high-quantile behavior. By contrast, negative values of \(\theta\) create a more alternating environment and have a smoothing effect on the workload process.

The signs of the induced correlations depend on the ordering of the state-dependent parameters. In the present parametrization, increasing \(\theta\) makes \(Corr(S_n,A_{n+1})\) negative and \(Corr(C_n,D_{n+1})\) positive. This is consistent with the fact that the state-dependent means of the corresponding random variables are ordered differently across the two states. The important point is that these correlations are generated solely by the Markovian dependence, since conditional independence is preserved given the states of the background chain. Finally, the numerical identification of the transform constants remains stable across all values of \(\theta\). The condition numbers of the corresponding linear systems range from \(9.93\) to \(13.54\), while the residual norms are of order \(10^{-16}\). This confirms that the moment computations remain numerically reliable throughout the experiment.

Figure~\ref{fig} describes the way normalized quantities of interest are affected for increasing values of $\theta$, in order to investigate the impact of modulation. Each plotted quantity is normalized by its value at $\theta=0$, which corresponds to the case with no serial dependence in the background chain (e.g., $\frac{E(W;\theta)}{E(W;\theta=0)}$). Since the stationary distribution of the chain is kept fixed, the marginal mixture distributions of the primitive random variables are unchanged; only the temporal dependence of the background environment varies. We can observe that positive persistence increases the stationary workload characteristics. The effect is most pronounced for the variance, which increases by approximately \(5.7\%\) (since $\frac{Var(W;\theta=0.6)-Var(W;\theta=0)}{Var(W;\theta=0)}\approx 0.057$) as \(\theta\) increases from \(0\) to \(0.9\). The mean and the simulated \(0.95\)-quantile also increase, while the coefficient of variation increases more mildly. Conversely, negative values of \(\theta\), corresponding to a more alternating environment, lead to smaller normalized values. This confirms that Markov modulation affects not only the average workload but also the variability and the right tail of the stationary distribution
\begin{figure}[ht!]
    \centering
    \includegraphics[width=0.8\linewidth]{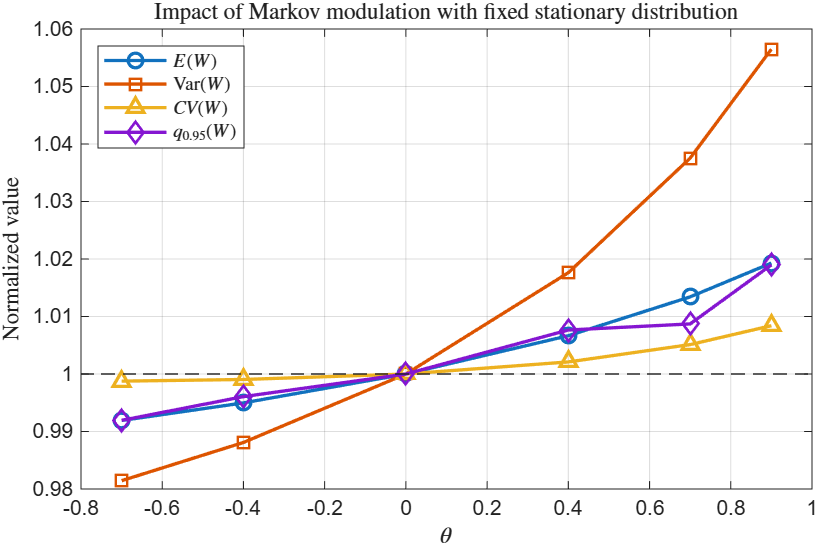}
    \caption{The effect of Markov modulation with fixed stationary distribution $\tilde{\pi}$}
    \label{fig}
\end{figure}

Therefore, the transform solution method we provided should not be viewed only as an analytical expression for moments, but instead, as a practical computational methodology by applying the following steps. This is the main computational value of the proposed approach:
\begin{enumerate}
    \item Markov-modulated model.
    \item Vector-valued transform equation.
    \item Finite-dimensional identification of constants.
    \item Stationary LST and Laplace inversion
\item Obtain $F_W(x)$, $P(W>x)$, and quantiles.
\end{enumerate}
\subsection{Example 2: The effect of Autocorrelation and Cross-correlation}

We now investigate the combined effect of autocorrelation and cross-correlation in
the spirit of the work in \cite{adan2}. We now consider a four-state background chain, exponential conditional distributions, and compare
systems with the same marginal mixture of interarrival and service times but different
dependence structures. In particular, we distinguish positive, negative and
zero cross-correlation between the state-dependent interarrival and service means, and
compare cases with and without autocorrelation in the interarrival and service-time sequences.

In the numerical experiment we set $p=0.6$, $q=0.4$, and use a four-state background Markov chain, $N=4$. Conditional on the background
states, we assume,
\begin{displaymath}
    \begin{array}{c}
         A_{n+1}\mid Z_{n+1}=j \sim exp(\lambda_j),
    \qquad
    D_{n+1}\mid Z_{n+1}=j \sim exp(\mu_j),\vspace{2mm}\\
    S_n\mid Z_n=i \sim exp(1/\gamma_i),
    \qquad
    C_n\mid Z_n=i \sim exp(1/\eta_i),
    \end{array}
\end{displaymath}
so that
\[
    E(S_n\mid Z_n=i)=\gamma_i,
    \qquad
    E(C_n\mid Z_n=i)=\eta_i.
\]
The parameters are chosen as
\begin{displaymath}
    \begin{array}{rl}
    (\lambda_1,\lambda_2,\lambda_3,\lambda_4)
    =&
    (0.10,\,10.00,\,0.12,\,9.80),\vspace{2mm}\\
    (\mu_1,\mu_2,\mu_3,\mu_4)
    =&
    (5.00,\,4.00,\,5.20,\,3.80),\vspace{2mm}\\
    (\eta_1,\eta_2,\eta_3,\eta_4)
    =&
    (0.25,\,0.25,\,0.25,\,0.25).
\end{array}
\end{displaymath}
The nominal load parameter is defined by
\[
    \rho_{\mathrm{nom}}
    =
    \frac{\sum_{i=1}^4 \pi_i \gamma_i}
         {\sum_{i=1}^4 \pi_i \lambda_i^{-1}},
\]
where $\tilde{\pi}^{T}=(\pi_{1},\pi_{2},\pi_{3},\pi_{4})$ is the stationary distribution of the background chain. Since $p<1$,
$\rho_{\mathrm{nom}}$ is not a stability condition for Model II, but used only as a
load parameter along which the service-time scale is varied. More precisely, for each
cross-correlation scenario we set
\[
    \gamma_i=u\gamma_i^{(0)},\qquad i=1,2,3,4,
\]
where \(u\) is chosen so that the prescribed value of \(\rho_{\mathrm{nom}}\) is obtained.

Following \cite[Remark 2.2]{adan2}, the cross-correlation is measured by the
correlation between the state-dependent interarrival and service means,
\begin{equation}
    \rho_{AS}
    =
    \frac{
    \sum_{i=1}^4
    \pi_i(\lambda_i^{-1}-\bar a)(\gamma_i^{(0)}-\bar \gamma)
    }
    {
    \left[
    \sum_{i=1}^4 \pi_i(\lambda_i^{-1}-\bar a)^2
    \sum_{i=1}^4 \pi_i(\gamma_i^{(0)}-\bar \gamma)^2
    \right]^{1/2}
    },\label{cross}
\end{equation}
where
\[
    \bar a=\sum_{i=1}^4\pi_i\lambda_i^{-1},
    \qquad
    \bar\gamma=\sum_{i=1}^4\pi_i\gamma_i^{(0)}.
\]
Thus, positive, negative and zero cross-correlation refer to the sign of
\(\rho_{AS}\), not to a lag-one correlation between \(S_n\) and \(A_{n+1}\).

We consider the following three service-mean patterns:
\begin{displaymath}
    \begin{array}{rl}
         \gamma^{(0)}=&(2,1,2,1),
    \quad \text{positive cross-correlation,}  \\
         \gamma^{(0)}=&(1,2,1,2),
    \quad \text{negative cross-correlation,}\\
    \gamma^{(0)}=&(1.9096,2.0000,1.0076,1.0828),
    \quad \text{zero cross-correlation.}
    \end{array}
\end{displaymath}
The last vector is chosen so that the cross-correlation in \eqref{cross} is zero under the stationary distribution used in the experiment.

For each cross-correlation scenario, we compare two transition matrices. The first one
induces autocorrelation in the selected interarrival-rate sequence
\(\{\lambda_{Z_n}\}\), and hence in the interarrival-time sequence \(\{A_n\}\). It is an
aperiodic perturbation of the cyclic matrix,
\[
    P_{\mathrm{ac}}
    =
    (1-\varepsilon)
    \begin{pmatrix}
        0&1&0&0\\
        0&0&1&0\\
        0&0&0&1\\
        1&0&0&0
    \end{pmatrix}
    +
    \varepsilon \frac{1}{4}{\bf e}{\bf e}^T,
    \qquad
    \varepsilon=0.02.
\]
The second one has identical rows,
\[
    P_{\mathrm{iid}}
    =
    \frac{1}{4}{\bf e}{\bf e}^T,
\]
so that the background states are independent over time and the interarrival-time
sequence has no autocorrelation.

For the zero cross-correlation case, the autocorrelated transition matrix is instead
based on
\[
    P_{0}
    =
    \begin{pmatrix}
        0&0.5&0&0.5\\
        0.5&0&0.5&0\\
        0&0.5&0&0.5\\
        0.5&0&0.5&0
    \end{pmatrix},
\]
and we use
\[
    P_{\mathrm{ac},0}
    =
    (1-\varepsilon)P_0
    +
    \varepsilon \frac{1}{4}{\bf e}{\bf e}^T.
\]
This construction gives autocorrelated interarrival times while preserving the zero
cross-correlation structure.

For each value of \(\rho_{\mathrm{nom}}\), we identify the transform constants by solving
the finite-dimensional linear system associated with Theorem \ref{th1},
and then compute $E(W)$, $Var(W)$, $CV(W)=\frac{\sqrt{Var(W)}}{E(W)}$ by using the moment recursion of Theorem \ref{th5}. Figures  \ref{fig:pos}, \ref{fig:neg}, and \ref{fig:zero} show the stationary mean \(E(W)\), the variance $Var(W)$ and the
coefficient of variation $CV(W)$ under  positive, negative and zero cross-correlation, respectively.

In the positive cross-correlation case, the two curves are much closer. The mean
workload and the variance are almost unaffected by the presence of autocorrelation,
especially when compared with the negative cross-correlation case. The coefficient of
variation is also nearly identical in the two cases and decreases as
\(\rho_{\mathrm{nom}}\) increases. This suggests that, under positive cross-correlation,
the marginal pairing of large interarrival means with large service means already
dominates the behaviour, and the additional temporal autocorrelation of the background
chain has only a limited effect.

The negative cross-correlation case exhibits the strongest impact of autocorrelation.
When autocorrelation is present, both $E(W)$ and $Var(W)$ are
substantially smaller than in the case without autocorrelation. This indicates that the
temporal ordering induced by the Markovian environment can mitigate congestion when
large interarrival means are paired with small service means, and vice versa. In this
case, the coefficient of variation is also smaller under autocorrelation over the whole
range of \(\rho_{\mathrm{nom}}\). Hence, autocorrelation reduces not only the mean
workload and the absolute variability, but also the variability relative to the mean.

In the zero cross-correlation case, autocorrelation in the interarrival-rate environment
has a visible but moderate effect. The mean workload is lower when autocorrelation is
present. The same qualitative behaviour is observed for the variance: the autocorrelated
case produces smaller absolute dispersion than the case without autocorrelation. The
coefficient of variation decreases as \(\rho_{\mathrm{nom}}\) increases in both cases,
and the two curves remain close. Thus, when the state-dependent interarrival and
service means are essentially uncorrelated, autocorrelation mainly affects the level and
absolute spread of the workload, while its effect on relative variability is weak.

Overall, these experiments show that the effect of autocorrelation cannot be understood without taking into account the cross-correlation structure between interarrival and
service times. Autocorrelation has the strongest impact under negative cross-correlation, where it substantially reduces the mean and variance of the stationary workload. Under
zero cross-correlation its effect is moderate, while under positive cross-correlation it is comparatively weak. Moreover, the variance and the coefficient of variation provide complementary information: $Var(W)$ measures absolute dispersion,
whereas $CV(W)$ measures dispersion relative to the mean. In all three scenarios,
$CV(W)$ decreases as the nominal load increases, which indicates that the mean
workload grows faster than the standard deviation over the considered parameter range.

\begin{figure}
    \centering
    \includegraphics[width=0.25\linewidth]{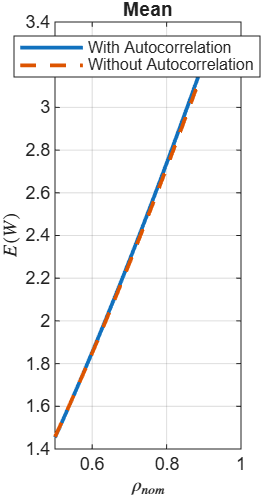}\quad\includegraphics[width=0.23\linewidth]{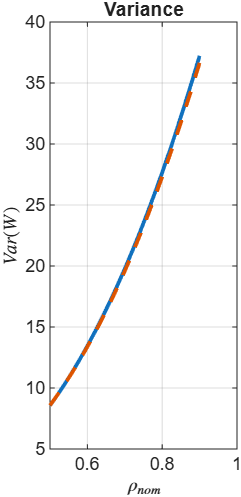}\quad\includegraphics[width=0.25\linewidth]{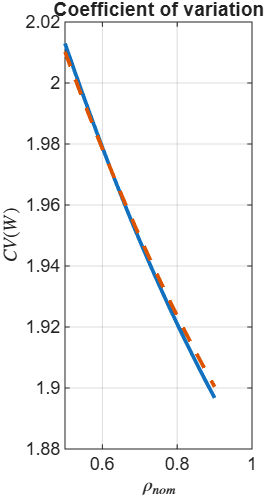}
    \caption{Positive cross-correlation case. The effect of autocorrelation is weak:
    the mean, variance and coefficient of variation are close to those obtained in
    the case without autocorrelation.}
    \label{fig:pos}
\end{figure}
\begin{figure}
    \centering
    \includegraphics[width=0.25\linewidth]{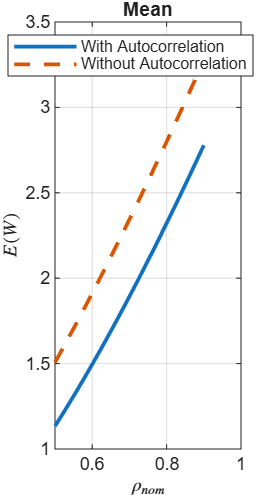}\quad\includegraphics[width=0.23\linewidth]{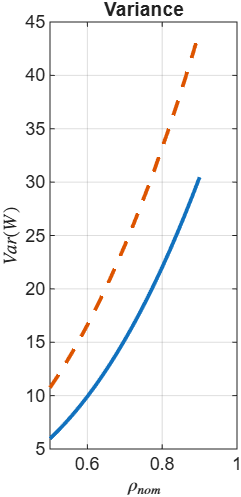}\quad\includegraphics[width=0.25\linewidth]{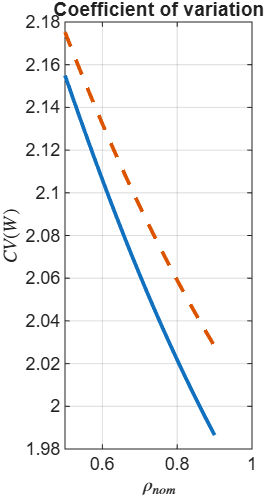}
    \caption{Negative cross-correlation case. Autocorrelation in the background
    environment substantially reduces both the stationary mean and variance, and
    also leads to a smaller coefficient of variation.}
    \label{fig:neg}
\end{figure}
\begin{figure}
    \centering
    \includegraphics[width=0.31\linewidth]{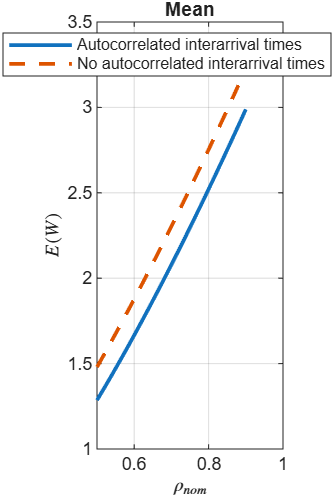}\quad\includegraphics[width=0.23\linewidth]{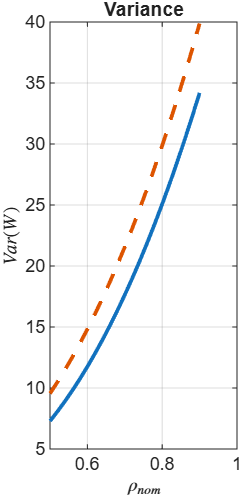}\quad\includegraphics[width=0.25\linewidth]{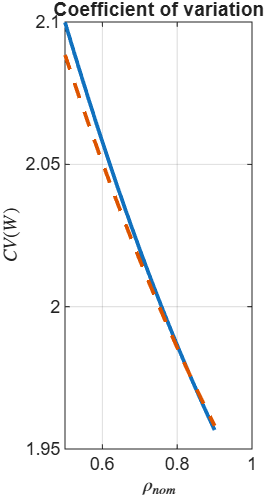}
    \caption{Zero cross-correlation case. Stationary mean, variance and coefficient
    of variation as functions of $\rho_{\mathrm{nom}}$, with and without
    autocorrelation in the background environment.}
    \label{fig:zero}
\end{figure}
\section{Concluding remarks and suggestions for future research}\label{conc}
In this work, we focused on the stationary analysis of a Markov-modulated version of a multiplicative Lindley-type recursion of the form $W_{n+1}=[V_{n}W_{n}+Y_{n}(V_{n})]^{+}$. We explicitly analyzed two cases: a) Model I: $V$ equals either 1 with probability $p_{1}$, or equals $a\in(0,1)$ with probability $p_{2}$, or it is negative with probability $p_{3}=1-p_{1}-p_{2}$, and $Y_{n}(V_{n})=S_{n}-A_{n+1}$, where $A_{n}$, $S_{n}$ have a rational LST. b) Model II: $V$ equals either one with probability $p\in(0,1)$ or minus one with probability $q=1-p$, and $Y_{n}(V_{n})$ have a general form and are also written as a difference of two nonnegative random variables that are dependent both on $V$. In both cases, $Y_{n}(V_{n})$ depend also on the state of an irreducible discrete-time Markov chain with finite state space.

The analysis of Model I resulted in a vector-valued functional equation for the LST of the distribution of $W$ that was solved recursively, by using the Wiener-Hopf boundary value theory. In Model II, we come up with a semi-Markov queueing system where the interarrival and service times depend linearly and randomly on the waiting times. The analysis results in a vector-valued functional equation that we explicitly solved. We provided explicit stationary characteristics and a recursive approach to obtain the steady-state moments. We also investigated its asymptotic behaviour. Moreover, we further generalized Model II by considering the case where given the values of the background Markov chain, the components of $Y_{n}(V_{n})$ are also dependent based on the FGM copula. For such a case, we also provided explicit results for the LST of the waiting time distribution.

Cases that might allow explicit analysis are, for example: a) Allowing $D_{n}$ to follow a mixed Erlang distribution will not complicate the analysis in Model II. b) %Maybe an alternative approach is needed in order to analyse Model II when $D_{n}$ are deterministic random variables. d) 
%The case where $V_{n}$ are uniformly distributed seems to be a challenging problem that may result in a vector-valued integral equation. e) 
The analysis in Model II might be extended to the case where $V$ takes any finite number of negative values, i.e., when $P(V=1)=p$, and for $i = 1,...,n$, $P(V = -u_{i}) = p_i$, where $u_i > 0$ and $\sum_{i=1}^{n}p_i = 1 - p$. c) The case where $V_{n}$ are uniformly distributed seems to be a challenging problem that results in a Volterra-type vector-valued integral equation.

\section*{Acknowledgements}
The author is grateful to the Editors and the anonymous Reviewers for the insightful remarks and the very careful reading, which helped to improve the original exposition. 
The author gratefully acknowledges the Empirikion Foundation, Athens, Greece (\href{https://www.empirikion.gr/}{www.empirikion.gr}) for the financial support of this work. This work is dedicated to the memory of Antonis Kontis and Thanasis Mamalis.


\appendix
\section{Proof of Theorem \ref{theo1}}\label{prooftheo1}
Using the recursion \eqref{recu}, and the identity $e^{-s [x]^{+}}+e^{-s [x]^{-}}=e^{-s x}+1$ (where $[x]^{+}:=max\{x,0\}$, $[x]^{-}:=min\{x,0\}$), we have for $Re(s)=0$, the following equation for the transforms $\Phi_{W,j}^{n}(s)$, $j=1,\ldots,N$:
\begin{displaymath}
    \begin{array}{l}
        \Phi_{W,j}^{n+1}(s)= E\left(e^{-sW_{n+1}}1_{\{Z_{n+1}=j\}}\right)=\sum_{i=1}^{N}P(Z_{n}=i)E\left(e^{-sW_{n+1}}1_{\{Z_{n+1}=j\}}|Z_{n}=i\right)  \vspace{2mm}\\
         = \sum_{i=1}^{N}P(Z_{n}=i)E\left(e^{-s[V_{n}W_{n}+S_{n}-A_{n+1}]^{+}}1_{\{Z_{n+1}=j\}}|Z_{n}=i\right)\vspace{2mm}\\
         = \sum_{i=1}^{N}P(Z_{n}=i)p_{i,j}E\left(e^{-s[V_{n}W_{n}+S_{n}-A_{n+1}]^{+}}|Z_{n+1}=j,Z_{n}=i\right)\vspace{2mm} \\
         =  \sum_{i=1}^{N}P(Z_{n}=i)p_{i,j}E\left(e^{-s[V_{n}W_{n}+S_{n}-A_{n+1}]} +1-e^{-s[V_{n}W_{n}+S_{n}-A_{n+1}]^{-}}|Z_{n+1}=j,Z_{n}=i\right)\vspace{2mm} \\
         =\sum_{i=1}^{N}p_{i,j}E\left(e^{-sV_{n}W_{n}}1_{\{Z_{n}=i\}}\right)E\left(e^{-s(S_{n}-A_{n+1})}|Z_{n+1}=j,Z_{n}=i\right)          +U^{n}_{j}(s)\vspace{2mm}\\
         =\sum_{i=1}^{N}p_{i,j}H_{i,j}(s)\left[p_{1} \Phi_{W,i}^{n}(s)+p_{2} \Phi_{W,i}^{n}(as)+p_{3}\int_{-\infty}^{0} \Phi_{W,i}^{n}(sy)P(V_{n}^{-}\in dy)\right]
         +U^{n}_{j}(s),
    \end{array}
\end{displaymath}
where $V_{n}^{-}=\left(V_{n}|V_{n}<0\right)$, and
\begin{displaymath}\begin{array}{c}
    U^{n}_{j}(s):=\sum_{i=1}^{N}P(Z_{n}=i)p_{i,j}E\left(1-e^{-s[V_{n}W_{n}+S_{n}-A_{n+1}]^{-}}|Z_{n+1}=j,Z_{n}=i\right).\end{array}
\end{displaymath}
Letting $n\to\infty$, substituting $H_{i,j}(s)$ and rearranging we come up with the following functional equation for $Re(s)=0$:
\begin{equation}
    \begin{array}{l}
        \Phi_{W,j}(s)D_{A,j}(-s)\prod_{k=1}^{N}D_{B,k}(s)\vspace{2mm}\\-\sum_{i=1}^{N}p_{i,j}N_{Y}^{(i,j)}(s)\prod_{\nu\neq i}^{N}D_{B,\nu}(s) \left[ p_{1}\Phi_{W,i}(s)+p_{2}\Phi_{W,i}(as) \right]\vspace{2mm}\\
         =p_{3}\sum_{i=1}^{N}p_{i,j}N_{Y}^{(i,j)}(s)\prod_{\nu\neq i}^{N}D_{B,\nu}(s)\int_{-\infty}^{0} \Phi_{W,i}(sy)P(V^{-}\in dy)\vspace{2mm}\\+D_{A,j}(-s)\prod_{k=1}^{N}D_{B,k}(s)U_{j}(s),\,j=1,\ldots,N,
    \end{array}\label{basic}
\end{equation}
where $U_{j}(s)=\lim_{n\to\infty}U^{n}_{j}(s)$. Now the following are true:
\begin{enumerate}
    \item the left-hand side of \eqref{basic} is analytic in $Re (s) >0$ and continuous in $Re (s) \geq 0$,
    \item the right-hand side of \eqref{basic} is analytic in $Re (s) <0$ and continuous in $Re (s) \leq 0$,
\item for large $s$, and $j=1,\ldots,N,$ both sides are $\mathcal{O}(s^{l_{j}+\sum_{k=1}^{N}m_{k}})$ in their respective half-planes.
\end{enumerate}
Both sides in \eqref{basic} are well defined at the boundary $Re (s )= 0$, and the determination of the unknown 
functions $\Phi_{W,j}(s)$ and $U_{j}(s)$ under the conditions 1., 2., and 3. is reduced to the solution of a
Wiener–Hopf boundary value problem; see \cite{cohen}. By introducing a function
$K(s)$ that is equal to the left-hand side of \eqref{basic} for $Re (s) \geq 0$ and to the right-hand
side of \eqref{basic} for $Re (s)\leq 0$, we have a function that is analytic in the whole $s$-plane, and that for large $s$ is $\mathcal{O}(s^{l_{j}+\sum_{k=1}^{N}m_{k}})$. Then, Liouville’s theorem \cite[p. 85]{tit} states that  in their respective half-planes both
sides of \eqref{basic} are equal to the same $(s^{l_{j}+\sum_{k=1}^{N}m_{k}})-$th degree
polynomial in $s$. Thus, for $Re (s) \geq 0$,
\begin{equation}
    \begin{array}{l}
        \Phi_{W,j}(s)D_{A,j}(-s)\prod_{k=1}^{N}D_{B,k}(s)-\sum_{i=1}^{N}p_{i,j}N_{Y}^{(i,j)}(s)\prod_{\nu\neq i}^{N}D_{B,\nu}(s) \left[ p_{1}\Phi_{W,i}(s)\right.\vspace{2mm}\\\left.+p_{2}\Phi_{W,i}(as) \right]=\sum_{\omega=0}^{l_{j}+\sum_{k=1}^{N}m_{k}}c_{\omega,j}s^{\omega},
    \end{array}\label{a1}
\end{equation}
and, for $Re (s) \leq  0$,
\begin{equation}
    \begin{array}{l}
        p_{3}\sum_{i=1}^{N}p_{i,j}N_{Y}^{(i,j)}(s)\prod_{\nu\neq i}^{N}D_{B,\nu}(s)\int_{-\infty}^{0}  \Phi_{W,i}(sy)P(V^{-}\in dy)\vspace{2mm}\\+D_{A,j}(-s)\prod_{k=1}^{N}D_{B,k}(s)U_{j}(s)= \sum_{\omega=0}^{l_{j}+\sum_{k=1}^{N}m_{k}}c_{\omega,j}s^{\omega}.         
    \end{array}\label{a2}
\end{equation}

For $s=0$, in either \eqref{a1}, or \eqref{a2} yields $c_{0,j}=\pi_{j}p_{3}D_{A,j}(0)\prod_{k=1}^{N}D_{B,k}(0)$. Note that to fully specify $\tilde{\Phi}_{W}(s)$, we need to derive $c_{\omega,j}$, $\omega=1,2,\ldots,l_{j}+\sum_{k=1}^{N}m_{k}$, $j=1,\ldots,N$, i.e., we need to derive a total $\sum_{j=1}^{N}l_{j}+N\sum_{k=1}^{N}m_{k}$ equations for the unknown terms $c_{\omega,j}$. 

This task can be accomplished in the following steps:
%Setting $t=t_{f,k}$, $f=1,\ldots,m_{k}$, $k=1,\ldots,N,$ in \eqref{a2} we obtain a system of $N\sum_{k=1}^{N}m_{k}$ equations:
%\begin{equation}
 %   \begin{array}{l}
  %     p_{3}\sum_{i=1}^{N}p_{i,j}N_{Y}^{(i,j)}(t_{f,k})\prod_{\nu\neq i}^{N}D_{B,\nu}(t_{f,k})\int_{-\infty}^{0} \Phi_{W,j}(t_{f,k}y)P(V^{-}\in dy)\vspace{2mm}\\=\sum_{\omega=0}^{l_{j}+\sum_{k=1}^{N}m_{k}}c_{\omega,j}(t_{f,k})^{\omega},\,j=1,\ldots,N.
   % \end{array}
   % \label{sys1}
%\end{equation}
\begin{enumerate}
    \item First note that \eqref{a1} can be rewritten as \eqref{b0}. In matrix notation, \eqref{a1} (or equivalently, \eqref{b0}) results for $s\in\mathcal D_+$ in
\begin{equation}
    (I-p_{1}F(s))\tilde{\Phi}_{W}(s)=p_{2}F(s)\tilde{\Phi}_{W}(as)+\tilde{v}(s).
    \label{b1}
\end{equation}
For $s\in\Omega=\{s\in\mathcal D_+:det(I-p_{1}F(s))\neq 0\}$,
\begin{equation}
    \tilde{\Phi}_{W}(s)=R(s)\tilde{\Phi}_{W}(as)+\tilde{V}(s),\label{b2}
\end{equation}
where $R(s)=(I-p_{1}F(s))^{-1}p_{2}F(s)$, $\tilde{V}(s)=(I-p_{1}F(s))^{-1}\tilde{v}(s)$. 
\begin{remark}
Note that the set $\Omega$ is nonempty. Indeed, $H_{i,j}(0)=1$ for all $i,j$,
and $\Delta_j(0)\neq0$, so that $0\in\mathcal D_+$. Consequently,
$F(0)=P^T$. Since $P$ is a transition matrix,
$\rho(P^T)=1$, and hence
\[
\rho(p_1F(0))=p_1<1.
\]
Therefore, $I-p_1F(0)$ is invertible, and thus $0\in\Omega$.
This establishes non-emptiness of $\Omega$. Invertibility at a
general $s\in\mathcal D_+$ is required only when $s\in\Omega$.
\end{remark}

Iterating \eqref{b2} $n$ times yield
\begin{equation}
   \tilde{\Phi}_{W}(s)=\sum_{k=0}^{n-1}\prod_{m=0}^{k-1}R(a^{m}s)\tilde{V}(a^{k}s)+\prod_{m=0}^{n-1}R(a^{m}s)\tilde{\Phi}_{W}(a^{n}s).\label{conv}
\end{equation}
The next step is to investigate how \eqref{conv} behaves as $n\to\infty$.
%In proving the convergence of \eqref{conv} as $n\to\infty$ we use the following Lemma.
%\begin{lemma}\label{lemma1}
%    For $n=0,1,\ldots,$ and $p_{1}+p_{2}<1$, $||\prod_{m=0}^{n}R(a^{m}s)||\leq c_{s}\tau^{n}$, for $s\in\Omega$, with $\tau\leq 1$, $c_{s}<\infty$, and where $||C||=max_{i,j}|C_{i,j}|$ is a matrix norm of $C\in \mathbb{C}^{N\times N}$.
%\end{lemma}

%As $n\to\infty$, $\tilde{\Phi}_{W}(a^{n}s)\to\tilde{\pi}$ and it is readily seen that as $m\to\infty$, $R(a^{m}s)\to p_{2}(I-p_{1}P^{T})^{-1}P^{T}$. For $p_{1}+p_{2}<1$, note also that $||p_{2}(I-p_{1}P^{T})^{-1}P^{T}||\leq r<1$ (where $||C||=max_{i,j}|C_{i,j}|$ is a matrix norm of $C\in \mathbb{C}^{N\times N}$). Therefore, $p_{2}(I-p_{1}P^{T})^{-1}P^{T}$ is a convergent matrix, i.e., as $m\to\infty$, $[p_{2}(I-p_{1}P^{T})^{-1}P^{T}]^{m}\to O$ (where $O$ denotes the zero matrix). Therefore, letting $n\to\infty,$
\begin{theorem}\label{th6}
    For every $s\in\Omega_a$, where
\begin{displaymath}
    \Omega_a
=
\{s\in\Omega:a^m s\in\Omega
\text{ for all }m\geq0\},
\end{displaymath}
the solution to \eqref{funeq} is given by
\begin{equation}
     \tilde{\Phi}_{W}(s)=\sum_{k=0}^{\infty}\prod_{m=0}^{k-1}R(a^{m}s)\tilde{V}(a^{k}s).
    \label{sol}
\end{equation}
\end{theorem}
\begin{proof}
Consider \eqref{conv}. We show that, as $n \to \infty$, the remainder term vanishes and the series converges absolutely. This procedure is summarized in the following steps:

\medskip
\noindent
\textbf{Step 1: Bounding $R(a^m s)$ and the ordered products.}

Since $R(s)$ is analytic close to $0$, and $R(0)=p_{2}(I-p_{1}P^{T})^{-1}P^{T}$, it is locally Lipschitz near $0$. Thus, there exist $\delta_{R}>0$, $C_{R}>0$, such that 
\begin{displaymath}
    \|R(s)-R(0)\|_{1}\leq C_{R}|s|,\quad |s|\leq\delta_{R},
\end{displaymath}
thus,
\begin{displaymath}
    \|R(a^{m}s)-R(0)\|_{1}\leq C_{R}a^{m}|s|.
\end{displaymath}
Since $a^{m}s\to 0$, as $m\to\infty$, there exists $n_{0}$: $\forall n\geq n_{0}$, $|a^{n}s|\leq \delta_{R}$. Therefore,
\begin{displaymath}
    \sum_{m=n_{0}}^{\infty}\|R(a^{m}s)-R(0)\|_{1}\leq C_{R}|s|\sum_{m=n_{0}}^{\infty}a^{m}=C_{R}|s|\frac{a^{n_{0}}}{1-a}<\infty,
\end{displaymath}
and thus,
\begin{displaymath}
    \sum_{m=0}^{\infty}\|R(a^{m}s)-R(0)\|_{1}<\infty.
\end{displaymath}
Note that 
\begin{displaymath}
    \|R(a^{m}s)\|_{1}=\|R(0)+R(a^{m}s)-R(0)\|_{1}\leq \|R(0)\|_{1}+\|R(a^{m}s)-R(0)\|_{1}.
\end{displaymath}
Now, by using the Neumann expansion,
\begin{displaymath}
    R(0)=p_{2}(I-p_{1}P^{T})^{-1}P^{T}=p_{2}\sum_{l=0}^{\infty}p_{1}^{l}(P^{T})^{l+1},
\end{displaymath}
and since $\|(P^{T})^{l+1}\|_{1}=1$,
\begin{displaymath}
    \|R(0)\|_{1}\leq p_{2}\sum_{l=0}^{\infty}p_{1}^{l}=r:=\frac{p_{2}}{1-p_{1}}<1,
\end{displaymath}
since we have assumed $p_{1}+p_{2}<1$. Choose any number $\tau$, such that $r<\tau<1$, and since $\|R(a^{m}s)-R(0)\|_{1}\to 0$, as $m\to\infty$, thus, there exists $m_{0}>0$, such that for $m\geq m_{0}$, we have $\|R(a^{m}s)-R(0)\|_{1}\leq \epsilon$, with $\epsilon>0$. Let $\epsilon:=\tau-r$. Then, for $m\geq m_{0}$,
\begin{displaymath}
\begin{array}{rl}
    \|R(a^{m}s)\|_{1}\leq& \|R(0)\|_{1}+\|R(a^{m}s)-R(0)\|_{1}\leq \frac{p_{2}}{1-p_{1}}+\|R(a^{m}s)-R(0)\|_{1}\vspace{2mm}\\
    \leq& r+(\tau-r)=\tau<1.
\end{array}
\end{displaymath}
With that in mind, for $k\geq m_{0}$
\begin{displaymath}
    \begin{array}{rl}
        \|\prod_{m=0}^{k-1}R(a^{m}s)\|_{1}\leq & \|\prod_{m=0}^{m_{0}-1}R(a^{m}s)\|_{1}  \|\prod_{m=m_{0}}^{k-1}R(a^{m}s)\|_{1}\leq C_{s}\prod_{m=m_{0}}^{k-1}\|R(a^{m}s)\|_{1} \vspace{2mm}\\
         \leq & C_{s}\tau^{k-m_{0}}=C^{\prime}_{s}\tau^{k},
    \end{array}
\end{displaymath}
where $C_{s}:=\|\prod_{m=0}^{m_{0}-1}R(a^{m}s)\|_{1}$. Thus the ordered products decay geometrically.
\medskip

\noindent
\textbf{Step 2: Bounding $\tilde V(a^k s)$.}

For $s\in\Omega_a$, we have $a^k s\to0$ as $k\to\infty$.
The functions $F$ and $\tilde v$ are analytic in a neighborhood
of $0$, and
\[
\tilde v(0)=p_3\pi,\quad F(0)=P^T.
\]
Since $P^T\pi=\pi$, it follows that
\[
\tilde V(0)
=(I-p_1P^T)^{-1}p_3\pi
=\frac{p_3}{1-p_1}\pi\neq0.
\]
Consequently, as $k\to\infty$
\[
\tilde V(a^k s)\rightarrow
\frac{p_3}{1-p_1}\pi.
\]
Thus, for each fixed $s\in\Omega_a$, the sequence
$\{\tilde V(a^k s):k\geq0\}$ is bounded. In particular, there
exists a finite constant $C_V(s)$ such that
\[
\|\tilde V(a^k s)\|_1\leq C_V(s),\quad k\geq0.
\]

\medskip

\noindent
\textbf{Step 3: Convergence of the series}

Using Steps 1 and 2, we get
\begin{displaymath}
    \|\prod_{m=0}^{k-1}R(a^{m}s)\tilde{V}(a^k s)\|_{1}\leq \|\prod_{m=0}^{k-1}R(a^{m}s)\|_{1}\|\tilde{V}(a^k s)\|_{1}\leq  C^{\prime}_{s}C_{V}(s)\tau^{k},
\end{displaymath}
and since $\tau<1$,
\begin{displaymath}
    \sum_{k=0}^{\infty} \|\prod_{m=0}^{k-1}R(a^{m}s)\tilde{V}(a^k s)\|_{1}<\infty,
\end{displaymath}
hence, the series converges.
\medskip

\noindent
\textbf{Step 4: Vanishing of the remainder term.}

In the following, we show that the term
\begin{displaymath}
    \mathcal{R}_{n}(s):=\prod_{m=0}^{n-1}R(a^{m}s)\tilde{\Phi}_{W}(a^{n}s),
\end{displaymath}
vanishes as $n\to \infty$. In particular, we already know that
\begin{displaymath}
    \|\prod_{m=0}^{n-1}R(a^{m}s)\|_{1}\leq C_{s}^{\prime}\tau^{n}.
\end{displaymath}
Moreover, 
\begin{displaymath}
    |\Phi_{W,j}(a^ns)|\leq E(|e^{-a^nsW}|1_{\{Z=j\}})=E(e^{-a^nRe(s)W}1_{\{Z=j\}})\leq E(1_{\{Z=j\}})=\pi_{j}.
\end{displaymath}
Thus,
\begin{displaymath}
    \|\tilde{\Phi}_{W}(a^{n}s)\|_{1}=\sum_{j=1}^{n}|\Phi_{W,j}(a^ns)|\leq \sum_{j=1}^{n}\pi_{j}=1,
\end{displaymath}
and then,
\begin{displaymath}
    \|\mathcal{R}_{n}(s)\|_{1}\leq C_{s}^{\prime}\tau^{n}\|\tilde{\Phi}_{W}(a^{n}s)\|_{1}\leq C_{s}^{\prime}\tau^{n}\to 0,
\end{displaymath}
as $n\to\infty$. Therefore, the remainder term vanishes.

To conclude, by letting $n \to \infty$, we finally have
\[
\tilde \Phi_W(s)
=
\sum_{k=0}^{\infty}
\left(\prod_{m=0}^{k-1} R(a^m s)\right)\tilde V(a^k s),
\quad s\in\Omega_a.
\]
\end{proof}

We now have an expression for the $\tilde{\Phi}_{W}(s)$ that contains the unknown terms, and will be used to derive a linear system for these unknowns.
\item Consider \eqref{undiv}, and let $\bar\sigma:=
\min_{\substack{j=1,\ldots,N\\r=1,\ldots,\ell_j}}
\bigl\{-Re(s_{r,j})\bigr\}>0$. Let also \(q_A>0\) be a common exponential-moment boundary for the interarrival times. More precisely, \(E(e^{\sigma A_{n+1}}\mid Z_n=i,Z_{n+1}=j)<\infty\) for every transition whenever \(0\le \sigma<q_A\). Choose
\[
0<\epsilon<\min\{\bar\sigma,q_A\},
\quad
J:=\{s\in\mathbb C:0<Re(s)<\epsilon\}.
\]
Note that the zeros of \(D_{A,j}(-s)\) have real parts \(-Re(s_{r,j})\ge\bar\sigma\), while the zeros of \(D_{B,k}(s)\) lie in the left half-plane. Thus, the denominator factors do not vanish on \(J\). Under these conditions, \(J\subset\mathcal D_+\), \eqref{undiv} holds there and $\Delta_{j}(s)\neq 0$, $j\in E$.

Let, $\Delta(s):=diag(\Delta_1(s),\ldots,\Delta_N(s))$, and
\begin{displaymath}
    c_j(s):=\sum_{\omega=0}^{d_j}c_{\omega,j}s^\omega,
\qquad
\tilde c(s):=(c_1(s),\ldots,c_N(s))^T.
\end{displaymath}
Define also
\[
\mathcal Q(s):=\Delta(s)(I-p_1F(s)),
\qquad
\mathcal G(s):=\Delta(s)F(s).
\]
Under the rational-transform assumptions, \(\mathcal Q(s)\), \(\mathcal G(s)\), and
\(\tilde c(s)=\Delta(s)\tilde v(s)\) are polynomial in \(s\).
Premultiplying \eqref{undiv} by \(\Delta(s)\) gives for $s\in J$,
\begin{equation}
\mathcal Q(s)\tilde\Phi_W(s)
=p_2\mathcal G(s)\tilde\Phi_W(as)+\tilde c(s).
\label{clearedFE}
\end{equation}
Note that both sides of \eqref{clearedFE} are analytic in
\(Re(s)>0\). Indeed, the matrix and vector coefficients are polynomial, and
\(\tilde\Phi_W(s)\) and \(\tilde\Phi_W(as)\) are analytic there.
Hence, by the identity theorem componentwisely
(see, e.g., \cite[Theorem~3.7 and Corollary~3.8, §IV.3]{Conway1978}), \eqref{clearedFE} holds throughout
\(Re(s)>0\). The transforms are continuous on the imaginary axis, so
\eqref{clearedFE} also holds for \(Re(s)=0\). Thus, \eqref{clearedFE} holds for $Re(s)\geq 0$.

We next count the zeros of \(\det\mathcal Q(s)\) in the right half-plane. We proceed by making the following transformation to apply \cite[Theorem 1, Appendix 2]{desmit}. Let \(\Pi=diag(\pi_1,\ldots,\pi_N)\), and since $\{Z_n,n\geq 0\}$ is irreducible, so that $\pi_i>0$, $i\in E$, set
\[
\bar F(s):=\Pi^{-1}F(s)\Pi.
\]
Since \(F_{ji}(s)=p_{i,j}H_{i,j}(s)\), we have for $\bar p_{ji}:=\frac{\pi_i p_{i,j}}{\pi_j}$ that
\[
\bar F_{ji}(s)=\bar p_{ji}H_{i,j}(s).
\]
The matrix \(\bar P=(\bar p_{ji})_{i,j\in E}\) is row-stochastic. Moreover,
\[
\Pi^{-1}\mathcal Q(s)\Pi
=\Delta(s)-p_1\Delta(s)\bar F(s).
\]
Apply \cite[Theorem~1, Appendix~2]{desmit} with
\[
B(s)=\Delta(s),
\quad
A(s)=-p_1\Delta(s)\bar F(s).
\]
For \(s=i\psi\), each \(H_{i,j}(i\psi)\) gives the value of a characteristic function, thus, \(|H_{i,j}(i\psi)|\leq1\). Hence, for every row \(j\in E\),
\[
\sum_{i=1}^N |A_{ji}(i\psi)|
=p_1|\Delta_j(i\psi)|
  \sum_{i=1}^N\bar p_{ji}|H_{i,j}(i\psi)|
\leq p_1|\Delta_j(i\psi)|
<|\Delta_j(i\psi)|
=|B_{jj}(i\psi)|.
\]

For the rational transforms considered here, \(H_{i,j}(s)\) has a finite
limit as \(|s|\to\infty\), \(Re(s)\geq0\), whose modulus is at most one. More precisely, the rational continuation of $H_{i,j}$ has the limit
\[
\lim_{|s|\to\infty}H_{i,j}(s)
=P(S_n=0\mid Z_n=i)
 P(A_{n+1}=0\mid Z_{n+1}=j)
\leq 1.
\]
Because there are finitely many pairs \(i,j\),
for every \(\eta>0\) there is \(R_0<\infty\) such that for all $i,j\in E$, $Re(s)\geq0$, $|s|\geq R_0$, we have
\[
|H_{i,j}(s)|\leq1+\eta.
\]
Choose \(\eta>0\) so that \(p_1(1+\eta)<1\). Then, on every
sufficiently large right semicircle,
\[
\sum_{i=1}^N |A_{ji}(s)|
\leq p_1(1+\eta)|\Delta_j(s)|
<|\Delta_j(s)|=|B_{jj}(s)|.
\]
Thus, the conditions of \cite[Theorem~1, Appendix~2]{desmit} hold on the boundary of a sufficiently large right half-disc.

Moreover, the elements of \(\mathcal Q\) and \(\Delta\) are polynomials, so their determinants have no poles. Also,
\[
det(\Delta(s))
=\left(\prod_{j=1}^N D_{A,j}(-s)\right)
 \left(\prod_{k=1}^N D_{B,k}(s)\right)^N.
\]
The zeros of \(D_{B,k}(s)\) lie in the left half-plane, while
\(D_{A,j}(-s)\) has \(l_j\) zeros in the right half-plane, counted with
multiplicity. Then, by following \cite[Theorem~1, Appendix~2]{desmit}, $det(\mathcal Q(s))=0$ has total multiplicity $L:=\sum_{j=1}^Nl_j$ in $Re(s)>0$. For the coefficient-determination procedure that follows, we adopt the following condition:
\begin{equation}
det(\mathcal Q(\delta))=0,\quad Re(\delta)>0\Rightarrow
\frac{d}{ds}det(\mathcal Q(s))|_{s=\delta}\neq0.
\label{st}
\end{equation}
Under \eqref{st}, there are \(L\) distinct roots in the right half-plane. Label them, in any fixed order, as
\(\delta_{r,j}\), \(r=1,\ldots,l_j\), \(j=1,\ldots,N\). The
double index is only an enumeration of these \(L\) roots. Under \eqref{st}, each zero $\delta_{r,j}$ of $det(\mathcal Q(s))$ with $Re(\delta_{r,j})>0$ is simple. Hence
$\mathcal Q(\delta_{r,j})$ has rank $N-1$.
Therefore, there exist nonzero row vectors $\tilde{\zeta}_{r,j}$, $r=1,\ldots,l_{j}$, $j=1,\ldots,N$, such that 
\begin{equation}\begin{array}{c}
    \tilde{\zeta}_{r,j}\mathcal{Q}(\delta_{r,j})=0.
\end{array}
    \label{bvz}
\end{equation}
%Following \cite[Theorem 1, Appendix 2]{desmit} (see also \cite{adan2}, and Appendix %\ref{roots}), we can show that $det(I-p_{1}F(s))=0$ has exactly $\sum_{j=1}^{N}l_{j}$ simple roots in
%$\mathcal D_+\cap\{s:Re(s)>0\}$. Denote these zeros by
%$\delta_{r,j}$, $r=1,\ldots,l_j$, $j=1,\ldots,N$. In particular, %$\delta_{r,j}\in\mathcal Z_{+}=\{ s\in \mathcal D_{+},det(I-p_1F(s))=0\}$. 

Substituting $s=\delta_{r,j}$, $r=1,\ldots,l_{j}$, $j=1,\ldots,N$ in \eqref{clearedFE}, pre-multiplying the resulting equation with $\tilde{\zeta}_{r,j}$, and having in mind \eqref{bvz} results in
\begin{equation*}
    \tilde{\zeta}_{r,j}\mathcal{Q}(\delta_{r,j})\tilde{\Phi}_{W}(\delta_{r,j})=\tilde{\zeta}_{r,j}[p_{2}\mathcal{G}(\delta_{r,j})\tilde{\Phi}_{W}(a\delta_{r,j})+\tilde{c}(\delta_{r,j})]=0,
\end{equation*}
or equivalently,
\begin{equation}
    p_{2}\tilde{\zeta}_{r,j}\mathcal{G}(\delta_{r,j})\tilde{\Phi}_{W}(a\delta_{r,j})=-\tilde{\zeta}_{r,j}\tilde{c}(\delta_{r,j}).\label{sys2}
\end{equation}
\begin{remark}
The simplicity condition \eqref{st} is adopted to state the coefficient
determination in terms of one root-cancellation equation at each root.
The transform equation \eqref{clearedFE} itself does not require this
condition. We adopt this condition to enhance the readability. However, the case of multiple zeros of $det(\mathcal Q(s))$ can be also treated.

Suppose that $\delta$ is a zero of $det(\mathcal Q(s))$ of multiplicity
$r_\delta\geq 1$, and define
\[
b(s):=p_2\mathcal G(s)\tilde{\Phi}_W(as)+\tilde c(s).
\]
In a neighborhood of $\delta$, both $b(s)$ and
$\tilde{\Phi}_W(s)$ are analytic. At points where $\mathcal Q(s)$
is nonsingular, \eqref{clearedFE} gives
\[
\tilde{\Phi}_W(s)
=
\frac{adj(\mathcal Q(s))b(s)}
     {det(\mathcal Q(s))}.
\]
Analyticity of the stationary transform at $\delta$ therefore requires
the corresponding singularity in this quotient to be removable. Equivalently, for $q=0,\ldots,r_\delta-1$,
\[
\frac{d^q}{ds^q}
\left[adj(\mathcal Q(s))b(s)\right]|_{s=\delta}
=0.
\]
These are the higher-order root-cancellation conditions. For a simple
zero, $r_\delta=1$, they reduce to the usual compatibility condition
used in \eqref{sys2}. For a multiple zero, the derivative conditions
may be coupled or linearly dependent, so the rank of the resulting
coefficient system must be checked.
\end{remark}

Note that coefficient equations \eqref{sys2} obtained at the zeros of
$\det\mathcal Q(s)$ involve the values
$\tilde\Phi_W(a\delta_{r,j})$. Theorem \ref{th6} gives the series
representation \eqref{sol} on $\Omega_a$, but the root count for the
kernel $det(\mathcal{Q}(s))$ that we use does not ensure that $a\delta_{r,j}\in\Omega_a$ (or even that $a\delta_{r,j}\in\mathcal D_+$). In particular, if \(s\in\Omega_a\), the original series \(\eqref{sol}\) applies directly.
If \(s\in\mathcal D_+\setminus\Omega_a\),  we interpret
$\tilde{\Phi}_W(a\delta_{r,j})$ by the limiting procedure described in Remark~2: choose $s_n\in\Omega_a$ with $s_n\to a\delta_{r,j}$,
evaluate the full series at each $s_n$, and then take the limit. More precisely, for $s_n\in\Omega_a$,
\begin{displaymath}
    \tilde\Phi_W(a\delta_{r,j})=\lim_{n\to\infty}
\sum_{k=0}^{\infty}\left(\prod_{m=0}^{k-1}R(a^m s_n)\right)
\tilde V(a^k s_n),\quad s_n\to a\delta_{r,j}.
\end{displaymath}
If \(s\notin\mathcal D_+\), neither the original series nor Remark 2 justifies the value. Thus, \eqref{sol} cannot
in general be applied directly at these points. The following
corollary extends the series representation, by using \eqref{clearedFE} and the same contraction near zero as in
Step~2 of the proof of Theorem~\ref{th6}.

\begin{corollary}\label{cleared-series}
Consider the functional equation \eqref{clearedFE}, which holds for $Re(s)\geq0$. Define
\[
\Omega_a^{\mathcal Q}
:=
\left\{s\in\mathbb C:Re(s)\geq0,\,det(\mathcal Q(a^m s))\neq 0,\,m\geq0\right\}.
\]
For every $s\in \Omega_a^{\mathcal Q}$, set $R_{\mathcal Q}(s):=
p_2\mathcal Q(s)^{-1}\mathcal G(s)$, $\tilde V_{\mathcal Q}(s):=
\mathcal Q(s)^{-1}\widetilde c(s)$. Then, for every $s\in\Omega_a^{\mathcal Q}$,
\begin{equation}
\tilde\Phi_W(s)
=
\sum_{k=0}^{\infty}
\left(\prod_{m=0}^{k-1}R_{\mathcal Q}(a^m s)\right)
\widetilde V_{\mathcal Q}(a^k s).
\label{clearedseries}
\end{equation}
If $s_\star\in\overline{\mathbb C}_+\setminus\Omega_a^{\mathcal Q}$,
where $\overline{\mathbb C}_+:=\{s\in\mathbb C:Re(s)\geq0\}$,
then its transform value is obtained by taking a limit through
$\Omega_a^{\mathcal Q}$:
\[
\tilde\Phi_W(s_\star)
=
\lim_{n\to\infty}
\sum_{k=0}^{\infty}
\left(\prod_{m=0}^{k-1}R_{\mathcal Q}(a^m s_n)\right)
\tilde V_{\mathcal Q}(a^k s_n),
\qquad
s_n\in\Omega_a^{\mathcal Q},\quad s_n\to s_\star.
\]
For each $n$, the full series is evaluated at $s_n$ before the
limit is taken.
\end{corollary}
\begin{proof}
    The proof follows the lines of that of Theorem \ref{th6}, as well as the arguments used in Remark 2, but for completeness we briefly give the basic steps.

Since $\Delta(0)$ is nonsingular and $F(0)=P^T$, we have $\mathcal Q(0)=\Delta(0)(I-p_1P^T)$, which is nonsingular because $p_1<1$ and $\rho(P^T)=1$. Moreover,
\[
R_{\mathcal Q}(0)
=p_2(I-p_1P^T)^{-1}P^T,
\qquad
\|R_{\mathcal Q}(0)\|_1
\leq \frac{p_2}{1-p_1}<1.
\]
Here we use that $P^T$ is column-stochastic and $p_1+p_2+p_3=1$
with $p_3>0$. Thus, by continuity, $R_{\mathcal Q}$ is contractive
in a neighborhood of zero, and $\widetilde V_{\mathcal Q}$ is
bounded there.

For $s\in\Omega_a^{\mathcal Q}$, multiply \eqref{clearedFE} by $\mathcal Q(s)^{-1}$ and iterate. For every $n\geq1$,
\begin{displaymath}
    \tilde\Phi_W(s)=\sum_{k=0}^{n-1}
\left(\prod_{m=0}^{k-1}R_{\mathcal Q}(a^m s)\right)
\tilde V_{\mathcal Q}(a^k s)+
\left(\prod_{m=0}^{n-1}R_{\mathcal Q}(a^m s)\right)
\tilde\Phi_W(a^n s).
\end{displaymath}
For fixed $s$, all sufficiently late iterates $a^m s$ lie in the
neighborhood of zero where $R_{\mathcal Q}$ is contractive.
The initial factors form a fixed finite product, and
$\widetilde\Phi_W$ is bounded on $\overline{\mathbb C}_+$.
Consequently, the remainder tends to zero and the series converges
absolutely. This proves \eqref{clearedseries}.

It remains to justify the limiting statement. Since
$det(\mathcal Q(0))\ne0$, there is a neighborhood of zero in which
$det(\mathcal Q(s))$ has no zeros. In any bounded neighborhood of
$s_\star$, all sufficiently late iterates $a^m s$ lie in this
zero-free neighborhood. Only finitely many earlier iterates need to
avoid zeros of $det(\mathcal Q(s))$. These zeros are isolated because
$det(\mathcal Q(s))$ is a nonzero polynomial. Hence points of
$\Omega_a^{\mathcal Q}$ can be chosen arbitrarily close to
$s_\star$. At each such point the series equals
$\tilde\Phi_W$, and continuity of the stationary transform on
$\overline{\mathbb C}_+$ gives the stated limit.
\end{proof}

Wherever \(\Delta(s)\) is invertible, $R_{\mathcal Q}(s)=R(s)$, $\tilde V_{\mathcal Q}(s)=\widetilde V(s)$. The corollary extends the representation to points where the original factors may be undefined. In particular, it allows you to evaluate \(\tilde\Phi_W(a\delta_{r,j})\) even if an iterate $a\delta_{r,j}$ is a zero of a denominator factor. Only a zero of \(det(\mathcal Q(s))\) requires taking the stated limit. 

Thus, in either case, \eqref{sys2} gives a system of $\sum_{j=1}^{N}l_{j}$ linear equations for the
unknown coefficients $c_{\omega,j}$.
\item The remaining $N\sum_{k=1}^{N}m_{k}$ are obtained as follows: Define $D_B(s):=\prod_{k=1}^{N}D_{B,k}(s)$, and $\mathcal T_B:=\{\tau\in\mathbb C:D_B(\tau)=0\}$. For each distinct zero $\tau\in\mathcal T_B$ denote by $r_{\tau}$ its multiplicity as a zero of the product $D_B$. Thus,
\begin{displaymath}
    \sum_{\tau\in\mathcal T_B}r_\tau=\sum_{k=1}^{N}m_k,
\end{displaymath}
where zeros are counted with multiplicity. This setting allows different
polynomials $D_{B,k}$ to have common zeros and also allows repeated zeros. For convenience, define
\begin{displaymath}
  \mathcal T_j(s):=p_3\sum_{i=1}^{N}p_{i,j}N_Y^{(i,j)}(s)\prod_{\nu\ne i}D_{B,\nu}(s)
\int_{-\infty}^{0}\Phi_{W,i}(sy)P(V^-\in dy).  
\end{displaymath}

Thus, \eqref{a2} is rewritten as
\begin{displaymath}
    \mathcal T_j(s)+D_{A,j}(-s)D_B(s)U_j(s)=c_j(s),\quad Re(s)\leq 0.
\end{displaymath}
For each $\tau\in\mathcal T_B$, all terms in this identity are analytic in a
neighborhood of $\tau$. Moreover, $Re(\tau)<0$, whereas the zeros of
$D_{A,j}(-s)$ lie in the right half-plane (i.e., $D_{A,j}(-\tau)\ne0$). The term
$D_{A,j}(-s)D_B(s)U_j(s)$ therefore has a zero of order at least $r_\tau$ at $s=\tau$.

It follows that, for every $j=1,\ldots,N$, every
$\tau\in\mathcal T_B$, and every $q=0,\ldots,r_\tau-1$,
\begin{equation}
\mathcal T_j^{(q)}(\tau)=c_j^{(q)}(\tau)
=\sum_{\omega=q}^{d_j}
c_{\omega,j}\frac{\omega!}{(\omega-q)!}\tau^{\omega-q},
\label{sys10}
\end{equation}
where $g^{(q)}(\eta)$ denotes the $q$th derivative of $g$ at point $\eta$. The derivatives in \eqref{sys10} are with respect to $s$ and apply to the
entire function $\mathcal T_j(s)$, including every term in the sum and
the integral. In evaluating $\Phi_{W,i}(sy)$ in that integral, we use the
representation in Corollary~\ref{cleared-series}, with its stated limiting interpretation
at exceptional points.

These conditions give $N\sum_{k=1}^{N}m_k$ scalar cancellation
conditions, counted with multiplicity, before any linearly dependent
conditions are removed. If $\tau$ is a simple zero of exactly one
factor, say $D_{B,k}$, and not a zero of any other $D_{B,\nu}$, $\nu\neq k$ then
$r_\tau=1$ and \eqref{sys10} reduces to
\begin{displaymath}
    p_3p_{k,j}N_Y^{(k,j)}(\tau)
\prod_{\nu\ne k}D_{B,\nu}(\tau)
\int_{-\infty}^{0}\Phi_{W,k}(\tau y)\,
\mathbb P(V^-\in dy)
=c_j(\tau),
\end{displaymath}
which is the equation obtained by direct substitution in \eqref{a2}.
If denominator factors share a zero, the derivative conditions in
\eqref{sys10} provide the corresponding multiplicity conditions.
Together with \eqref{sys2}, these equations form the system used to
determine the coefficients $c_{\omega,j}$, $\omega\geq1$.

 Thus, the coefficients determine $\tilde{\Phi}_W(s)$: for
$s\in\Omega_a$, it is given by \eqref{sol}. More generally, for
$s\in\Omega_a^{\mathcal Q}$, it is given by \eqref{clearedseries}.
At points of $\mathbb C_+\setminus\Omega_a^{\mathcal Q}$, its value is
obtained by the limiting procedure in Corollary~\ref{cleared-series}. This finishes the proof.
\end{enumerate}

\section{Proof of Theorem \ref{theo}}\label{roots}
 The proof follows the lines in \cite{adan2}. We first assume that $\beta_{i}^{*}(s)$ are analytic for all $s$, $Re(s)>-\zeta$, for some $\zeta>0$. Consider the circle $C_{\epsilon}$ with center $\widehat{\lambda}=max\{\lambda_{1},\ldots,\lambda_{N}\}$ and radius $\epsilon+\widehat{\lambda}$, $0<\epsilon<\zeta$. Hence, for
each $i$,
\begin{displaymath}
    |s-\lambda_i|
\geq \bigl||s-\widehat{\lambda}|-
          |\widehat{\lambda}-\lambda_i|\bigr|
=\lambda_i+\epsilon.
\end{displaymath}
Let $G_{\epsilon}(s,u):=uH(s)+sI-\Lambda$, and $M(s,u):=upB^{*}(s)P+s\Lambda^{-1}-I$, $u\in[0,1]$, so that $G_{\epsilon}(s,u)=M(s,u)\Lambda$. 
 
We first prove that $det(G_{\epsilon}(s,u))\neq 0$, $u\in[0,1]$, $s\in C_{\epsilon}$, $Re(s)\geq 0$. Thus, it suffices to show that $M(s,u)$ is nonsingular. Since
$|\beta_i^*(s)|\leq1$ for $Re(s)\geq0$, we have 
\begin{displaymath}
    |M_{ii}(s,u)|
=|\frac{s}{\lambda_i}-1+up\,p_{ii}\beta_i^*(s)|
\geq 1+\frac{\epsilon}{\lambda_i}-up\,p_{ii}.
\end{displaymath}
Moreover,
\begin{displaymath}
    \sum_{j\ne i}|M_{ij}(s,u)|
=up|\beta_i^*(s)|\sum_{j\ne i}p_{ij}
\leq up(1-p_{ii}).
\end{displaymath}
Therefore,
\begin{displaymath}
    |M_{ii}(s,u)|-\sum_{j\ne i}|M_{ij}(s,u)|
\geq 1+\frac{\epsilon}{\lambda_i}-up>0,
\end{displaymath}
since $up\leq p<1$. Therefore, $M(s,u)$ is strictly diagonally
dominant by rows, and $det (G_\epsilon(s,u))=det (M(s,u))det(\Lambda)\ne0$ for $s\in C_\epsilon$, $Re(s)\geq0$, and $u\in[0,1]$.

 We now prove that $det (G_\epsilon(s,u))=det (M(s,u))det(\Lambda)\ne0$  for $s\in C_{\epsilon}$, $Re(s)< 0$. Note first that at $s=0$, $M(0,u)=upP-I$, and since $P$ is stochastic, $up\leq p<1$, $upP-I$ is invertible for every $u\in[0,1]$. Thus, by continuity and compactness of $[0,1]$,
there is a $\delta>0$ such that $M(s,u)$ remains nonsingular whenever
$|s|\leq\delta$ and $u\in[0,1]$. Now we have to show that we can always find small $\epsilon>0$ such that the left-hand arc of the circle $C_{\epsilon}$ arc falls inside that neighborhood. Indeed, for \(s\in C_\epsilon\), $|s-\widehat\lambda|=\widehat\lambda+\epsilon$. Then,
\begin{displaymath}
    \begin{array}{rl}
       |s-\widehat\lambda|^2=&(\widehat\lambda+\epsilon)^2\Leftrightarrow\\
         |s|^2 =&2\widehat\lambda Re(s)+2\widehat\lambda\epsilon+\epsilon^2\leq2\widehat\lambda\epsilon+\epsilon^2.
    \end{array}
\end{displaymath}
Now choose $0<\epsilon<\zeta$ sufficiently small, such that $\sqrt{2\widehat\lambda\epsilon+\epsilon^2}<\delta$. Then
$M(s,u)$ is nonsingular on this part of $C_\epsilon$ for
every $u\in[0,1]$. Equivalently, $G_{\epsilon}(s,u)$ is non singular for $s\in C_{\epsilon}$, $Re(s)< 0$, $u\in[0,1]$.

Since $\epsilon<\zeta$, the closed disk $D_\epsilon$ bounded by
$C_\epsilon$ is contained in
$\{Re(s)\geq-\epsilon\}\subset\{Re(s)>-\zeta\}$. Hence
$M(\cdot,u)$ is analytic on a neighborhood of $\overline{D_\epsilon}$
for each $u\in[0,1]$. Together with the fact that $det (M(s,u))\neq 0$ on $C_\epsilon$ established above, this justifies
applying the argument principle to $\det M(s,u)$.

Let $f(u)$ be the number of zeros of $det(upB^{*}(s)P+s\Lambda^{-1}-I)$ inside $C_{\epsilon}$:
\begin{displaymath}
    f(u)=\frac{1}{2\pi i}\int_{C_{\epsilon}}\frac{\frac{\partial}{\partial s}det(upB^{*}(s)P+s\Lambda^{-1}-I)}{det(upB^{*}(s)P+s\Lambda^{-1}-I)}ds.
\end{displaymath}
Clearly $f(0)=N$, since $det(s\Lambda^{-1}-I)=\prod_{i=1}^{N}(\frac{s}{\lambda_{i}}-1)$, and thus, the zeros are $s_{i}=\lambda_{i}$ and exactly $N$ have $Re(s)\geq 0$, and thus, are inside $C_{\epsilon}$. Since $f(u)$ is an integer-valued continuous function on $[0,1]$, it is constant, therefore, $f(1)=N$.

By the
argument-principle above, $det (M(s,1))$ has exactly $N$
zeros in $D_\epsilon$, counted with multiplicity. We first show that
none of these zeros has nonpositive real part. If
$s\in\overline{D_\epsilon}$ and $Re(s)\leq0$, then
\[
|s-\widehat{\lambda}|^2\leq(\widehat{\lambda}+\epsilon)^2
\]
implies
\[
|s|^2
\leq 2\widehat{\lambda}Re(s)
      +2\widehat{\lambda}\epsilon+\epsilon^2
\leq 2\widehat{\lambda}\epsilon+\epsilon^2
<\delta^2.
\]
By the preceding continuity argument, $M(s,u)$ is nonsingular for
$|s|\leq\delta$ and every $u\in[0,1]$. Therefore, for these $s$,
$M(s,1)$ is nonsingular. Hence all $N$ zeros in $D_\epsilon$ have
positive real parts.

Conversely, suppose that $Re(s)\geq0$ and $M(s,1)x=0$ for some
nonzero vector $x$. Choose $i$ such that
$|x_i|=\max_j|x_j|$. The $i$th row of this equation and
$|\beta_i^*(s)|\leq1$ give
\[
\left|\frac{s}{\lambda_i}-1\right|
\leq p.
\]
Consequently,
\[
|s-\widehat{\lambda}|
\leq |s-\lambda_i|+\widehat{\lambda}-\lambda_i
\leq p\lambda_i+\widehat{\lambda}-\lambda_i
=\widehat{\lambda}-(1-p)\lambda_i
<\widehat{\lambda}.
\]
Thus, every zero with nonnegative real part lies inside the disk
bounded by $C_0$, and hence inside $D_\epsilon$. It follows that
$\det M(s,1)=0$, and equivalently
$\det(pB^*(s)P\Lambda+sI-\Lambda)=0$, has exactly $N$ zeros in
$Re(s)>0$, counted with multiplicity.

\section{Allowing additional dependence based on the FGM copula}\label{secfgm}
Now consider the case where given $Z_{n}$, and $Z_{n+1}$, the random variables $S_{n}$, $A_{n+1}$, and $D_{n+1}$, $C_{n}$ are not conditionally independent, but instead, they are dependent based on the FGM copula. In general, copulas are functions that couple multivariate distribution functions to their one-dimensional marginal distribution functions \cite{nelsen}. Thus, a copula itself is a multivariate distribution function whose inputs are the
respective marginal cumulative probability distribution functions for the random variables
of interest.

A bivariate copula $C$ is a joint distribution function on $[0,1]\times[0,1]$ with standard uniform marginal  distributions. Under Sklar's theorem \cite{nelsen}, any bivariate distribution function $F$ with marginals $F_1$ and $F_2$ can   be   written   as $F(x,y)=C(F_{1}(x),F_{2}(y))$, for some copula $C$. For more details on copulas and their properties see \cite{joe,nelsen}. Modeling the dependence structure between random variables (r.v.) using copulas has become popular in actuarial science and financial risk management; e.g. \cite{denoi,albteu,cossette1} (non-exhaustive list). In queueing literature there has been some recent works where copulas are used to introduce dependency. In \cite{trapa}, the authors applied a new constructed bound copula to analyze the dependency between two
parallel service times. In \cite{lei}, the authors developed a new method based on copulas to model correlated inputs in discrete-event stochastic systems. Motivated by recent empirical evidence, the authors in \cite{wu1,wu2} used a fluid model with bivariate dependence orders and copulas to study the dependence among service times and patience times in large scale service systems; see also \cite{gu,wang,rai}. In our work we focus on a dependence structure based on the FGM copula, which is defined by
\begin{displaymath}
    C_{\eta}^{FGM}(u_{1},u_{2})=u_{1}u_{2}+\eta u_{1}u_{2}(1-u_{1})(1-u_{2}),\,(u_{1},u_{2})\in[0,1]\times[0,1],\,\eta\in[-1,1].
\end{displaymath}
The density function associated to the above expression is
\begin{displaymath}
    c_{\eta}^{FGM}(u_{1},u_{2}):=\frac{\partial^{2}}{\partial u_{1}\partial u_{2}}C_{\eta}^{FGM}(u_{1},u_{2})=1+\eta(1-2u_{1})(1-2u_{2}).
\end{displaymath}

The FGM copula allows for negative and positive dependence, and includes the independence copula ($\eta=0$), which corresponds to the basic model analyzed in Section \ref{mod2}. Our primary motivation for choosing an FGM copula is due to its simplicity and tractability thanks to its polynomial structure, which allows exact results. It has a simple analytical form, which is easy to apply, thus making it attractive when we are dealing with the modeling of bivariate dependent data. Its simple analytic shape enables closed-form solutions to many problems in applied probability. Moreover, it is a first order approximation of the Plackett copula and of the Frank copula \cite[p. 100, and p. 106, respectively]{nelsen}. For applications of FGM copula in risk theory, health insurance plans, financial risk management, stochastic frontiers and a stereological context, see e.g., \cite{cossette1,barges,mai} and the references therein.

In general, the bivariate distribution, say $F_{X,Y}$, of the bivariate random vector $(X,Y)$ with continuous marginals, say $F_{X}$, $F_{Y}$, and with a dependence structure based on the FGM copula is defined for $x,y\in\mathbb{R}^{+}$, $\eta\in[-1,1]$ by
\begin{displaymath}
\begin{array}{rl}
   F_{X,Y}(x,y)=&C_{\eta}^{FGM}(F_{X}(x),F_{Y}(y))\vspace{2mm}\\
     =& F_{X}(x)F_{Y}(y)+\eta F_{X}(x)F_{Y}(y)\left(1-F_{X}(x)\right)\left(1-F_{Y}(y)\right).
\end{array}
\end{displaymath}
 With this, the bivariate density of $(X,Y)$ is given by,
\begin{displaymath}
\begin{array}{rl}
   f_{X,Y}(x,y)=&c_{\eta}^{FGM}(F_{X}(x),F_{Y}(y))f_{X}(x)f_{Y}(y)\vspace{2mm}\\
     =& f_{X}(x)f_{Y}(y)+\eta g(x)f_{Y}(y)\left(2\bar{F}_{Y}(y)-1\right),\,x,y\in\mathbb{R}^{+},\,\eta\in[-1,1],
\end{array}
\end{displaymath}
where $g(x)=f_{X}(x)(1-2F_{X}(x))$ with Laplace transform $g^{*}(s)=\int_{0}^{\infty}e^{-sx}g(x)dx$.

In the following, we focus on the Markov-dependent reflected autoregressive process, and assume that $\{(S_{n},A_{n+1}),n\geq 1\}$, and $\{(C_{n},D_{n+1}),n\geq 1\}$ form a sequence of i.i.d. random vectors distributed as the canonical r.v. $(S,A)$, in which the components are dependent based on the FGM copula. To our best knowledge, there is no other work in the related literature that focus on (either scalar or Markov-dependent) reflected autoregressive processes with dependencies based on copulas. Thus, we are dealing with a recursion of the form in \eqref{bqq}, where now for $n\geq 0$, $x,y\geq 0$, $i,j=1,\ldots,N$:
\begin{equation}
    \begin{array}{l}
         P(A_{n+1}\leq x,S_{n}\leq y,Y_{n+1}=j|Y_{n}=i,A_{2},\ldots,A_{n},S_{1},\ldots,S_{n-1},Z_{1},\ldots,Z_{n-1}) \vspace{2mm}\\
        =  P(A_{n+1}\leq x,S_{n}\leq y,Y_{n+1}=j|Y_{n}=i)=p_{i,j}F_{S,A|i,j}(y,x),\vspace{2mm}\\
         P(D_{n+1}\leq x,C_{n}\leq y,Z_{n+1}=j|Z_{n}=i,A_{2},\ldots,C_{n},D_{1},\ldots,D_{n-1},Z_{1},\ldots,Z_{n-1}) \vspace{2mm}\\
        =  P(D_{n+1}\leq x,C_{n}\leq y,Z_{n+1}=j|Z_{n}=i)=p_{i,j}F_{C,D|i,j}(y,x),
    \end{array}\label{mp1m}
\end{equation}
where, $F_{S,A|i,j}(y,x)$, $F_{C,D|i,j}(y,x)$ are the bivariate distribution functions of $(S_{n},A_{n+1})$, $(C_{n},D_{n+1})$ given $Z_{n}$, $Z_{n+1}$, respectively, with marginals $F_{S,i}(y)$, $F_{A,j}(x)$, and $F_{C,i}(y)$, $F_{D,j}(x)$, defined as $F_{S,A|i,j}(y,x)=C_{\eta}^{FGM}(F_{S,i}(y),F_{A,j}(x))$, $F_{C,D|i,j}(y,x)=C_{\eta}^{FGM}(F_{C,i}(y),F_{D,j}(x))$ for $(y,x)\in \mathbb{R}^{+}\times \mathbb{R}^{+}$. The bivariate densities of $(S,A)$, $(C,D)$ are given by 
\begin{displaymath}
\begin{array}{rl}
     f_{S,A|i,j}(y,x)=& f_{S,i}(y)\lambda_{j}e^{-\lambda_{j}x}+\eta g_{i}(y)\left[2\lambda_{j}e^{-2\lambda_{j}x}-\lambda_{j}e^{-\lambda_{j}x}\right],\vspace{2mm}\\
     f_{C,D|i,j}(y,x)=&f_{C,i}(y)\mu_{j}e^{-\mu_{j}x}+\eta h_{i}(y)\left[2\mu_{j}e^{-2\mu_{j}x}-\mu_{j}e^{-\mu_{j}x}\right].
\end{array}
\end{displaymath}
with $g_{i}(y):=f_{S,i}(y)(1-2F_{S,i}(y))$, $h_{i}(y):=f_{C,i}(y)(1-2F_{C,i}(y))$ with Laplace transforms $g_{i}^{*}(s)=\int_{0}^{\infty}e^{-sy}g_{i}(y)dy$, $h_{i}^{*}(s)=\int_{0}^{\infty}e^{-sy}h_{i}(y)dy$.

Let $G^{*}(s):=diag(g^{*}_{1}(s),\ldots,g^{*}_{N}(s))$, and
\begin{displaymath}
    \bar{H}(s):=p[B^{*}(s)P\Lambda(2\Lambda-sI)-s\eta G^{*}(s)P\Lambda].
\end{displaymath}
Then, we have the following result.
\begin{theorem}\label{th2}
    The transform row vector $\tilde{\Phi}_{W}^{T}(s)$ satisfies
    \begin{equation}
    \begin{array}{rl}
       \tilde{\Phi}_{W}^{T}(s)[\bar{H}(s)-(sI-\Lambda)(sI-2\Lambda)]=&\bar{v}(s),
       \end{array} \label{funz}
    \end{equation}
    where, $\tilde{\Phi}_{W}^{T}(0)=\tilde{\pi}^{T}$, and
    \begin{displaymath}
        \bar{v}(s)=[s\bar{k}_{1}+s^{2}\bar{k}_{2}+s^{3}\bar{k}_{3}+s^{4}\bar{k}_{4}-q\tilde{\pi}4\Lambda^{2}M^{2}]((M+sI)(2M+sI))^{-1},
    \end{displaymath}
    with $\bar{k}_{l}:=(k_{1,l},\ldots,k_{N,l})$, $l=1,2,3,4$, where for $j=1,\ldots,N$,
\begin{displaymath}
    \begin{array}{rl}
         k_{j,1}=&2\lambda_{j}\mu_{j}^{2}(2v_{1,j}^{(1)}+v_{2,j}^{(1)})+2\lambda_{j}^{2}\mu_{j}(2v_{1,j}^{(-1)}+v_{2,j}^{(-1)})-6\lambda_{j}\mu_{j}(\lambda_{j}-\mu_{j})q\pi_{j},\vspace{2mm}\\
         k_{j,2}=&\mu_{j}(3\lambda_{j}-\mu_{j})(2v_{1,j}^{(1)}+v_{2,j}^{(1)})+\lambda_{j}(\lambda_{j}^{2}-3\mu_{j})(2v_{1,j}^{(-1)}+v_{2,j}^{(-1)})-(2\lambda_{j}^{2}+2\mu_{j}^{2}-9\lambda_{j}\mu_{j})q\pi_{j},\vspace{2mm}\\
         k_{j,3}=&(2\lambda_{j}-3\mu_{j})v_{1,j}^{(1)}+(\lambda_{j}-3\mu_{j})v_{2,j}^{(1)}+(2\mu_{j}-3\lambda_{j})v_{1,j}^{(-1)}+(\mu_{j}-3\lambda_{j})v_{2,j}^{(-1)}-3(\mu_{j}-\lambda_{j})q\pi_{j},\vspace{2mm}\\
         k_{j,4}=&v_{1,j}^{(-1)}+v_{2,j}^{(-1)}-v_{1,j}^{(1)}-v_{2,j}^{(1)}-q\pi_{j}.
    \end{array}
\end{displaymath}
    \end{theorem}
    \begin{proof}
     The proof is similar  to the one in Theorem \ref{th1}, but since it is the first time in the related literature that the basic components of the model are dependent based on a copula we provide details. As in Theorem \ref{th1}. the transform $\Phi_{W,j}^{n+1}(s):=E(e^{-sW_{n+1}}1_{\{Z_{n+1}=j\}})$ satisfies
    \begin{equation}
        \begin{array}{l}
            E(e^{-sW_{n+1}}1_{\{Z_{n+1}=j\}})\vspace{2mm}\\= pE(e^{-s[S_{n}-A_{n+1}+W_{n}]^{+}}1_{\{Z_{n+1}=j\}})+qE(e^{-s[D_{n+1}-C_{n}-W_{n}]^{+}}1_{\{Z_{n+1}=j\}}).
        \end{array}\label{kop}
    \end{equation}
    Then, we have that
    \begin{displaymath}
        \begin{array}{c}
            E(e^{-s[S_{n}-A_{n+1}+W_{n}]^{+}}1_{\{Z_{n+1}=j\}})=  \sum_{i=1}^{N}P(Z_{n}=i)p_{i,j}E(e^{-s[S_{n}-A_{n+1}+W_{n}]^{+}}|Z_{n+1}=j,Z_{n}=i)\vspace{2mm} \\
            =   \sum_{i=1}^{N}P(Z_{n}=i)p_{i,j}E(\int_{y=0}^{\infty}\int_{x=0}^{W_{n}+y}e^{-s(W_{n}+y-x)}f_{S,A|i,j}(y,x)dxdy\vspace{2mm}\\+\int_{y=0}^{\infty}\int_{x=W_{n}+y}^{\infty}f_{S,A|i,j}(y,x)dxdy|Z_{n}=i),\vspace{2mm}\\
            E(e^{-s[D_{n+1}-C_{n}-W_{n}]^{+}}1_{\{Z_{n+1}=j\}})=  \sum_{i=1}^{N}P(Z_{n}=i)p_{i,j}E(e^{-s[D_{n+1}-C_{n}-W_{n}]^{+}}|Z_{n+1}=j,Z_{n}=i)\vspace{2mm} \\
            =   \sum_{i=1}^{N}P(Z_{n}=i)p_{i,j}E(\int_{y=0}^{\infty}\int_{x=0}^{W_{n}+y}f_{D,C|i,j}(y,x)dxdy\vspace{2mm}\\+\int_{y=0}^{\infty}\int_{x=W_{n}+y}^{\infty}e^{-s(x-y-W_{n})}f_{D,C|i,j}(y,x)dxdy|Z_{n}=i).
        \end{array}
    \end{displaymath}
    After heavy but straightforward computations, we obtain
    \begin{displaymath}
        \begin{array}{rl}
            E(e^{-s[S_{n}-A_{n+1}+W_{n}]^{+}}1_{\{Z_{n+1}=j\}})=   & \sum_{i=1}^{N}p_{i,j}[\frac{\lambda_{j}}{\lambda_{j}-s}(\beta_{i}^{*}(s)-\frac{\eta sg_{i}^{*}(s)}{2\lambda_{j}-s} )\Phi_{W,i}^{n}(s)\vspace{2mm}\\
            &-\frac{s(\beta_{i}^{*}(\lambda_{j})-\eta g_{i}^{*}(\lambda_{j})}{\lambda_{j}-s}\Phi_{W,i}^{n}(\lambda_{j})-\frac{s \eta g_{i}^{*}(2\lambda_{j})}{2\lambda_{j}-s}\Phi_{W,i}^{n}(2\lambda_{j})], \\
            E(e^{-s[D_{n+1}-C_{n}-W_{n}]^{+}}1_{\{Z_{n+1}=j\}})=  &\sum_{i=1}^{N}p_{i,j}\left[P(Z_{n}=j)\right.\vspace{2mm}\\
            &\left.-s(\frac{c_{i}^{*}(\mu_{j})-\eta h_{i}^{*}(\mu_{j})}{\mu_{j}+s}\Phi_{W,i}^{n}(\mu_{j})+\frac{\eta}{2\mu_{j}+s}\Phi_{W,i}^{n}(2\mu_{j}))\right].
        \end{array}
    \end{displaymath}
    Letting $n\to\infty$ and substituting back in \eqref{kop}, we obtain after straightforward computations \eqref{funz}.
    \end{proof}

    Note that $\tilde{\Phi}^{T}_{W}(s)$ is fully specified when $4N^{2}$ unknown terms, namely $\Phi_{W,i}(\lambda_{j})$, $\Phi_{W,i}(\mu_{j})$, $\Phi_{W,i}(2\lambda_{j})$, $\Phi_{W,i}(2\mu_{j})$, $i,j=1,\ldots,N$ that are involved in the vectors $\tilde{k}_{l}$, $l=1,2,3,4$, are fully determined. We now need three steps to fully determine the $4N^{2}$ unknown terms. The procedure is similar to the one followed before. 
    \begin{enumerate}
        \item The following theorem gives the number and placement of the solutions of $det(\bar{H}(s)-(sI-\Lambda)(sI-2\Lambda))=0$. That result will provide us with $2N$ equations that invole the unknowns that arise in $\tilde{k}_{l}$, $l=1,2,3,4$.
\begin{theorem}\label{theoc}
    The equation
    \begin{equation}
    det(\bar{H}(s)-(sI-\Lambda)(sI-2\Lambda))=0,
\label{eigx}
    \end{equation}
    has exactly $2N$ zeros, say $s_{1},\ldots,s_{2N}$, with $Re(s_{i})>0$, $i=1,\ldots,2N$.
\end{theorem} 
\begin{proof}
    The proof is similar to the one in Theorem \ref{theo} and further details are omitted.
\end{proof}
Now we state the following theorem, the proof of which is similar to the one in Theorem \ref{theor1}, so further details are omitted.
\begin{theorem}\label{theor1x}
    Suppose that the $2N$ solutions $s_i$, $i=1,2,\ldots,2N,$ with $Re(s_i) > 0$ to equation \eqref{eigx} are distinct. Let $\bar{a}_i=(a_{i,1},\ldots,a_{i,N})^{T}$ be a nonzero
column vector satisfying
\begin{equation*}
    (\bar{H}(s_{i})-(s_{i}I-\Lambda)(s_{i}I-2\Lambda))\tilde{a}_{i}=\tilde{0},\,i=1,\ldots,2N,\label{eig22}
\end{equation*}
where $\tilde{0}$ is a row vector of zeros. Then, for $i=1,\ldots,2N$, $\bar{v}(s_{i})\bar{a}_{i}=0$, or equivalently,
\begin{equation}
        \sum_{k=1}^{N}v_{k}(s_{i})a_{i,k}=0,
   \label{n1st}
\end{equation}
where $\bar{v}(s)=(v_{1}(s),\ldots,v_{N}(s))$ as given in Theorem \ref{th2}.
\end{theorem}
   \item 
Substituting $s=\mu_{l}$, and $s=2\mu_{l}$, for $l=1,\ldots,N$ in the expanding version of \eqref{funz}, we obtain $2N^{2}$ equations that relate the unknown terms. More precisely, substituting $s=\mu_{l}$, $l=1,\ldots,N$ we have for $j,l=1,\ldots,N$:
\begin{equation}
    \begin{array}{l}
        \lambda_{j}p\sum_{i=1}^{N}p_{i,j} \Phi_{W,i}(\mu_{l})[(2\lambda_{j}-\mu_{l})\beta_{i}^{*}(\mu_{l})-\eta\mu_{l}g_{i}^{*}(\mu_{l})]-(\lambda_{j}-\mu_{l})(2\lambda_{j}-\mu_{l})\Phi_{W,j}(\mu_{l})\vspace{2mm}\\=\frac{1}{(\mu_{j}+\mu_{l})(2\mu_{j}+\mu_{l})}\left[\sum_{m=1}^{4}k_{j,m}(\mu_{l})^{m}-q\pi_{j}4(\lambda_{j}\mu_{j})^{2}\right],
    \end{array}\label{q22}
\end{equation}
while setting $s=2\mu_{l}$, $l=1,\ldots,N$ we have for $j,l=1,\ldots,N$:
\begin{equation}
    \begin{array}{l}
        \lambda_{j}p\sum_{i=1}^{N}p_{i,j} \Phi_{W,i}(2\mu_{l})[2(\lambda_{j}-\mu_{l})\beta_{i}^{*}(2\mu_{l})-2\eta\mu_{l}g_{i}^{*}(2\mu_{l})]\vspace{2mm}\\-(\lambda_{j}-2\mu_{l})2(\lambda_{j}-\mu_{l})\Phi_{W,j}(2\mu_{l})=\frac{\sum_{m=1}^{4}k_{j,m}(2\mu_{l})^{m}-q\pi_{j}4(\lambda_{j}\mu_{j})^{2}}{(\mu_{j}+2\mu_{l})2(\mu_{j}+\mu_{l})}.
    \end{array}\label{q23}
\end{equation}
\item 
We also need the $2N(N-1)$ equations. These are obtained from the expanded version of \eqref{funz} as follows: For each $j=1,\ldots,N$, substitute $s=\lambda_{k}$, $k\neq j$, $k=1,\ldots,N$. Then, we have for $j=1,\ldots,N$, $k\neq j$:
\begin{equation}
    \begin{array}{l}
        \lambda_{j}p\sum_{i=1}^{N}p_{i,j} \Phi_{W,i}(\lambda_{k})[(2\lambda_{j}-\lambda_{k})\beta_{i}^{*}(\lambda_{k})-\eta\lambda_{k}g_{i}^{*}(\lambda_{k})]-(\lambda_{j}-\lambda_{k})(2\lambda_{j}-\lambda_{k})\Phi_{W,j}(\lambda_{k})\vspace{2mm}\\=\frac{1}{(\mu_{j}+\lambda_{k})(2\mu_{j}+\lambda_{k})}\left[\sum_{m=1}^{4}k_{j,m}(\lambda_{k})^{m}-q\pi_{j}4(\lambda_{j}\mu_{j})^{2}\right].
    \end{array}\label{q24}
\end{equation}
This procedure provides $N(N-1)$ equations that involve the unknown terms. The remaining $N(N-1)$ equations are obtained by setting in the expanded version of \eqref{funz}, $s=2\lambda_{k}$, $k\neq j$, $k=1,\ldots,N$. Then, we have for $j=1,\ldots,N$, $j\neq k$:
\begin{equation}
    \begin{array}{l}
        \lambda_{j}p\sum_{i=1}^{N}p_{i,j} \Phi_{W,i}(2\lambda_{k})[2(\lambda_{j}-\lambda_{k})\beta_{i}^{*}(2\lambda_{k})-2\eta\lambda_{k}g_{i}^{*}(2\lambda_{k})]\vspace{2mm}\\-(\lambda_{j}-2\lambda_{k})2(\lambda_{j}-\lambda_{k})\Phi_{W,j}(2\lambda_{k})=\frac{\sum_{m=1}^{4}k_{j,m}(2\lambda_{k})^{m}-q\pi_{j}4(\lambda_{j}\mu_{j})^{2}}{(\mu_{j}+2\lambda_{k})2(\mu_{j}+\lambda_{k})}.
    \end{array}\label{q25}
\end{equation}
Then, \eqref{n1st}-\eqref{q25} constitute $4N^{2}$ linear equations for the $4N^{2}$ unknowns $\Phi_{W,i}(\lambda_{j})$, $\Phi_{W,i}(\mu_{j})$, $\Phi_{W,i}(2\lambda_{j})$, $\Phi_{W,i}(2\mu_{j})$, $i,j=1,\ldots,N$. Once these terms are determined, the vectors $\tilde{k}_{l}$, $l=1,2,3,4$ are known, and thus, $\bar{v}(s)$ is known, which in turn implies that $\tilde{\Phi}^{T}_{W}(s)$ is known. Therefore, we are able to repeat the procedure in subsection \ref{mom} and derive the steady-state moments of the waiting time for the Markov-modulated model in the presence of dependencies based on the FGM copula. 
\end{enumerate}




\bibliographystyle{abbrv}

\bibliography{mybibitem}


\end{document}